\documentclass[12pt,reqno]{amsart}
\usepackage[T1]{fontenc} 
\usepackage[margin=1.5in]{geometry}
\usepackage{graphicx}
\usepackage{amssymb,amsmath,amsthm,float,enumerate}
\usepackage[all]{xy}
\usepackage{multicol}
\usepackage{amsfonts,amsmath, tikz}
\usetikzlibrary{topaths,calc}
\usepackage{leftidx}
\usepackage{tikz}
\usetikzlibrary{topaths,calc}
\usepackage{mathrsfs}
\usepackage{lscape}
\usepackage{setspace}
\usepackage{wrapfig}
\usepackage[framemethod=TikZ]{mdframed}

\mdfdefinestyle{MyFrame}{linecolor=black,
    outerlinewidth=.5pt,
    roundcorner=10pt,
    innertopmargin=5pt,
    innerbottommargin=5pt,
    innerrightmargin=5pt,
    innerleftmargin=5pt,
    backgroundcolor=white}

\newtheorem{theorem}{Theorem}[section]
\newtheorem{proposition}[theorem]{Proposition}

\newtheorem{lemma}[theorem]{Lemma}

\newtheorem{claim*}{Claim}

\theoremstyle{definition}
\newtheorem{definition}[theorem]{Definition}
\newtheorem{definition*}{Definition}

\numberwithin{equation}{section}
\newtheorem{Figure}[theorem]{Figure}
\newtheorem{Example}[theorem]{Example}
\newtheorem{Question*}{Question} 

\usepackage[normalem]{ulem}
\usepackage{soul}

\begin{document}

%\doublespacing 
%\tableofcontents 
 
\title[star log symplectic surfaces]{A classification of star log symplectic structures on a compact oriented surface}

\author{Melinda Lanius}
\address{Department of Mathematics, University of Arizona, 617 North Santa Rita Avenue, Tucson, AZ 85721, USA}
\email{lanius@math.arizona.edu}

\begin{abstract} Given a compact oriented surface, we classify log Poisson bi-vectors whose degeneracy loci are locally modeled by a finite set of lines in the plane intersecting at a point. Further, we compute the Poisson cohomology of such structures and discuss the relationship between our classification and the second Poisson cohomology.  
\end{abstract}
\maketitle
\section{\bf Introduction} 

As is well-known, two non-degenerate Poisson - i.e. symplectic - structures on a compact connected surface $S$ are isomorphic precisely when they have the same de Rham cohomology class. In this work we give an analogous classification of a more general class of Poisson bi-vector on a surface which we call star log symplectic structures. Since every bi-vector on a surface $S$ is Poisson, we organize Poisson structures on $S$ by their degeneracy loci and, in fact, Vladimir Arnol'd introduced a hierarchy for these degeneracies (Appendix 14 C \cite{Arnold}). Olga Radko provided a classification when the Poisson bi-vector degenerated linearly along a curve \cite{Radko}. \emph{We will extend her classification to Poisson bi-vectors whose degeneracy loci are locally modeled by a finite set of lines in the plane intersecting at a point. Additionally, we will compute the Poisson cohomology of such structures $\pi$ and discuss how our classification relates to deformations of $\pi$.  }

\subsection{Classification.} Let $(S,\pi)$ be a Poisson structure on a surface $S$. The key idea to our approach is to not work with $\pi$, but instead to work with the corresponding symplectic form, even when $\pi$ is degenerate. Recall that $\pi$  induces a map $\pi^\sharp$ between $T^\ast S$ and $TS$. When $\pi$ is non-degenerate, it admits an inverse. 

\centerline{$\xygraph{!{<0cm,0cm>;<1cm, 0cm>:<0cm,1cm>::}
!{(1.25,0)}*+{T^\ast S}="b"
!{(4.8,0)}*+{{TS}}="c"
!{(1.75,.1)}*+{{}}="d"
!{(4.25,.1)}*+{{}}="e"
!{(1.75,-.1)}*+{{}}="f"
!{(4.25,-.1)}*+{{}}="g"
"d":"e"^{\pi^{\sharp}}
"g":"f"^{\omega^{\flat}=(\pi^{\sharp})^{-1}}}$}

\noindent This inverse map defines a symplectic form $\omega$ for $S$. Moser's argument (see Theorem 2.70 in \cite{Cannas}), a fundamental technique in symplectic geometry, provides a classification: two symplectic forms $\omega_0$ and $\omega_1$ on a compact connected surface $S$ are symplectomorphic if and only if they are cohomologous in degree 2 de Rham cohomology. The first main result of this paper is a classification similar in spirit to this symplectic case. We will elaborate on the meaning of certain terms in the exposition following the statement.

\begin{theorem}\label{thmclass} (Classification of star log symplectic surfaces). Let $(S,D)$ be a compact connected surface with star divisor $D$. Two log-symplectic forms $\omega_0$ and $\omega_1$ on $(S,D)$ are symplectomorphic if and only if they are cohomologous in the degree 2 Lie algebroid cohomology of the $b$-tangent bundle and $\omega_0$ and $\omega_1$ induce the same orientation of the $b$-tangent bundle. \end{theorem}

To understand this theorem, let us begin by considering a class of Poisson structures that have some degeneracy. We are not able to immediately work with a corresponding form $\omega$ because the map $\pi^\sharp$ will not have an inverse. The trick is to find a class of Poisson structure where we can replace the tangent bundle with a Lie algebroid \(\mathcal{A}\). We can view bi-vectors $\pi$ in this class as non-degenerate on $\mathcal{A}$ and define an associated symplectic form on $\mathcal{A}^*$. 

\centerline{$\xygraph{!{<0cm,0cm>;<1cm, 0cm>:<0cm,1cm>::}
!{(1.25,0)}*+{\mathcal{A}^\ast}="b"
!{(4.8,0)}*+{{\mathcal{A}}}="c"
!{(1.75,.1)}*+{{}}="d"
!{(4.25,.1)}*+{{}}="e"
!{(1.75,-.1)}*+{{}}="f"
!{(4.25,-.1)}*+{{}}="g"
"d":"e"^{\pi^{\sharp}}
"g":"f"^{\omega^{\flat}=(\pi^{\sharp})^{-1}}}$}

\noindent This isomorphism allows us to use Moser-type arguments and we can hope for a classification similar to the symplectic case. 

In Section \ref{SectionClassification} we will provide a more detailed description of this method, but for now we offer a sampling of relevant recent literature: Victor Guillemin, Eva Miranda, and Ana Rita Pires \cite{Guillemin01} carried out this procedure for a class of structure called $b$-Poisson. Geoffrey Scott proved versions of Moser's theorem in what he named the  $b^k$-setting and used it to establish this characterization for $b^k$-Poisson surfaces (See section 6.1 of \cite{Scott} and in particular Theorem 6.7). In \cite{Lanius}, we use Moser techniques to establish local normal forms for scattering-symplectic structures on manifolds of any even dimension. Eva Miranda and Arnau Planas \cite{Miranda} further use these Moser techniques to establish a version of Scott's classification that is equivariant under the action of a group. 

We will classify a new type of Poisson structure - which we call star log symplectic - on a surface. Work has been done in the setting where Poisson bi-vectors degenerate linearly on a normal crossing divisor $D$, i.e. a set of smooth hypersurfaces $Z\subset M$ that intersect transversely, see \cite{Gualtieri02, Guillemin01, LaniusPartitionable} or \cite{Radko02}. Given a surface $S$, we will examine Poisson bi-vectors that degenerate linearly on a set $D$ of smooth hypersurfaces $Z\subset S$ whose intersections are modeled by a finite set of lines in the plane intersecting at a point. Such a set $D$ is called a \emph{star divisor}.

To define the $b$-tangent bundle, the Lie algebroid we will use, is subtle and requires more work than in the case of  single, or two intersecting, lines. We address the challenge of defining the $b$-tangent bundle in this more general setting by introducing an adapted atlas, which we call an \emph{astral atlas}. The complete details of this construction are provided in Section  \ref{secat}. 

\begin{Figure}\label{fig1} Some local models \vspace{1ex}

\begin{center}{\bf star log  symplectic }

\hspace{-.75ex}$\overbrace{\hspace{76ex}}_{}$\end{center}\vspace{-1ex}

\begin{minipage}{.25\linewidth} {\bf  $b$ - symplectic }

\hspace{-2ex}$\overbrace{\hspace{18ex}}_{}$\end{minipage}\vspace{-1ex}

\noindent\begin{minipage}{.25\linewidth}\begin{center}

$\pi=x\dfrac{\partial}{\partial x}\wedge\dfrac{\partial}{\partial y}$\vspace{2ex}

\emph{linear degeneracy}\vspace{1.5ex}

\includegraphics[scale=.9]{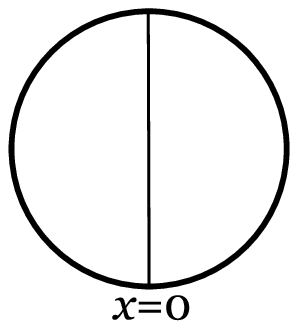}

 \end{center} 
\end{minipage}\hspace{2ex}\begin{minipage}{.33\linewidth} \begin{center}

$\pi=\lambda(x^2-y^2)\dfrac{\partial}{\partial x}\wedge\dfrac{\partial}{\partial y}$ \vspace{2ex}

\emph{quadratic degeneracy}  \vspace{1ex}

\includegraphics[scale=.9]{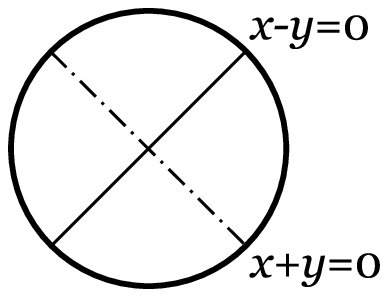}

\end{center}\end{minipage}\hspace{2ex}\begin{minipage}{.33\linewidth} \begin{center}

$\pi = \lambda(x^2y-y^3)\dfrac{\partial}{\partial x}\wedge\dfrac{\partial}{\partial y}$\vspace{2ex}

\emph{cubic degeneracy} \vspace{1ex}

\includegraphics[scale=.9]{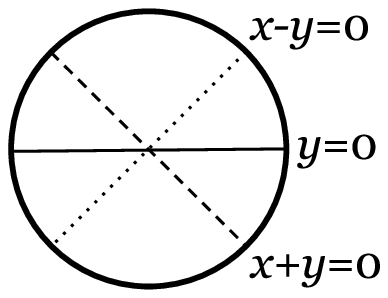}
\end{center}\end{minipage} \vspace{1ex}

 \begin{minipage}{.6\linewidth}\begin{center} \hspace{-5ex}$\underbrace{\hspace{44ex}}_{}$
 
 \hspace{-1.5ex}{\bf   normal crossing }
 \end{center}\end{minipage} 
\end{Figure}\vspace{1ex}

In Section \ref{SectionClassification} we classify star log symplectic surfaces up to $b$-symplectomorphism. For simplicity, we state our results here using a fixed tubular neighborhood of $D$. A more detailed and general discussion can be found in Section \ref{SectionClassification}. 

We realize star log symplectic Poisson bi-vectors as non-degenerate on Richard Melrose's $b$-tangent bundle \cite{MelroseGreenBook}. We will show that, given a surface with star divisor $(S,D=\left\{Z_1,\dots,Z_m\right\})$, the second Lie algebroid cohomology of the $b$-tangent bundle is $${}^bH^2(S)\simeq H^2(S)\oplus \bigoplus_{i}H^{1}(Z_i)\oplus\bigoplus_{i<j}H^{0}(Z_i\cap Z_j).$$ Note that this result is quite surprising since the Lie algebroid cohomology is only noticing pairwise intersection of curves. This is the cohomology featured in the classification Theorem \ref{thmclass}. Because the Lie algebroid cohomology only `sees' pair-wise intersection, we also require that two $b$-symplectic forms $\omega_0,\omega_1$ induce the same orientation of the $b$-tangent bundle, i.e. $\omega_0=f \omega_1$ for a positive function $f\in\mathcal{C}^\infty(S)$.

\subsection{Poisson Cohomology.} 

Next, we turn to Poisson cohomology. We state our result for star log symplectic surfaces $(S,D,\pi)$ here using a fixed tubular neighborhood of $D$. A more detailed discussion is given in Section \ref{section4}. 

\begin{theorem}\label{PC}(Poisson cohomology of a log Poisson structure).  Given a log-symplectic surface $(S,D=\left\{Z_1,\dots,Z_k\right\},\pi)$, i.e. a surface with a set $D$ of smooth hypersurfaces $Z_i\subset S$ whose intersections are modeled by a finite set of lines in the plane intersecting at a point, the Poisson cohomology in degrees 0 and 1 is $$H_\pi^0(S)\simeq H^0(S) \hspace{2ex} \text{ and } \hspace{2ex} H^1_\pi(S)\simeq H^1(S)\oplus \bigoplus_i H^0(Z_i).$$
The cohomology in degree 2 is a direct sum of the following vector spaces: 

\begin{enumerate}[i)]
    \item A single copy of $H^2(S)$.  
    \item For each hypersurfaces $Z_i$, a copy of $H^1(Z_i)$. 
    \item For each pair wise intersection of two hypersurfaces $Z_i,Z_j$ with $i<j$, $$\left[H^{0}(Z_i\cap Z_j)\right]^2$$
    \item For each intersection of three hypersurfaces $Z_i,Z_j,Z_k$ with $i<j<k$, $$\left[H^{0}(Z_i\cap Z_j\cap Z_k)\right]^3 $$ 
\item For each intersection of four or more hypersurfaces $Z_{i_1},Z_{i_2},Z_{i_3},\dots,Z_{i_\ell}$ with $i_1<i_2<i_3<\dots <i_\ell$, a copy of 
$$\left[H^{0}(Z_{i_1}\cap Z_{i_2}\cap \dots \cap Z_{i_\ell})\right]^4.$$
\end{enumerate}
\end{theorem}

Note that Theorem \ref{PC} generalizes previous work on
the Poisson cohomology of b-symplectic structures with normal
crossing divisors. In particular, we recover the known results for when $D$ consists of transversely intersecting curves. Olga Radko provides a description for disjoint curves \cite{Radko} and Nobutada Nakanishi completes the computation for the case of transversley intersecting curves \cite{Nakanishi}.

\subsection{Reconciling our classification and Poisson cohomology.} Recall that the second Poisson cohomology group of a Poisson structure $\pi$ allows us to describe formal infinitesimal deformations of $\pi$. In this subsection we will discuss how the various cohomology theories we employ throughout this paper sit in the theory of deformations of Lie algebroids at large and what insight they give into understanding deformations of Poisson structures.  

Let us begin by zooming all the way out to a general Lie algebroid $(\mathcal{A},[\cdot,\cdot],\rho)$.  Marius Crainic and Ieke Moerdijk's work \cite{Crainic} introduces a way to study deformations of the bracket $[\cdot,\cdot]$ using a cohomology theory aptly called deformation cohomology. We in particular are concerned with the Poisson Lie algebroid $(T^\ast M, [\cdot,\cdot], \pi^\sharp)$ where the Lie bracket is defined as $$[\alpha,\beta] = \mathcal{L}_{\pi^\sharp(\alpha)}(\beta)-\mathcal{L}_{\pi^\sharp(\beta)}(\alpha)-d(\pi(\alpha,\beta))$$ for all smooth de Rham 1-forms $\alpha, \beta$ on $M$. In Corollary 3 of \cite{Crainic}, Crainic and Moerdijk explain how a class in second degree Poisson cohomology gives a deformation of the Poisson Lie algebroid. In other words, Poisson cohomology measures deformations of the Lie bracket on $T^\ast M$ that remain Poisson, i.e. the bracket can always be defined by a Poisson bivector $\pi$ on $M$. 

In this paper we compute Poisson cohomology by constructing what we call the rigged Lie algebroid $\mathcal{R}$, which is a Lie algebroid isomorphic to the cotangent Lie algebroid $(T^*M,\pi^\sharp)$ of a Poisson manifold $(M,\pi)$. As we will show in detail in Section \ref{section4}, the anchor map of $\mathcal{R}$ is evaluation.  This perspective is more advantageus than working with $T^*M$ and the map $\pi^\sharp$  because we compute cohomology using a complex of forms with an exterior derivative, rather than multi-vector fields with the notoriously intractable Lichnerowicz differential.  

The $b$-de Rham forms, which we use to classify star log symplectic structures, form a subcomplex of the Lichnerowicz complex because they in fact form a subcomplex of $\mathcal{R}$-de Rham forms. The $b$-de Rham cohomology is related to an even more restricted type of infinitesimal deformation of the Poisson Lie algebroid: infinitesimal deformations of the bracket that can always be defined by a star log symplectic structure. 

\begin{center}\begin{minipage}{.7\textwidth} 

\scalebox{.9}{\xymatrix{  0 \ar[r]          & C^0_{def}(S) \ar[r]^\delta  & C^1_{def}(S)\ar[r]^\delta & C^2_{def}(S)\ar[r] & 0 \\ 0 \ar[r] & \mathcal{C}^\infty(S) \ar[u] \ar[r]^{\partial_\pi\hspace{2ex}} & \mathcal{C}^\infty(S,TS) \ar[u] \ar[r]^{\partial_\pi\hspace{1ex}} & \mathcal{C}^\infty(S,\wedge^2 TS) \ar[u] \ar[r] & 0 \\ 0 \ar[r]          & \ar[u] \mathcal{C}^\infty(S) \ar[r]^d  & {}^{b}\Omega^1(S) \ar[u] \ar[r]^d &{}^{b}\Omega^2(S) \ar[u] \ar[r] & 0 
} }

\end{minipage}\begin{minipage}{.26\textwidth} \tiny  \vspace{2ex}

\begin{flushleft} 
{\bf Deformation complex}

\vspace{8ex}

{\bf Lichnerowicz} {\bf complex}

  \vspace{8ex}

{\bf $b$-forms} \vspace{1ex}

\end{flushleft}
\end{minipage}\end{center}\vspace{2ex}  

In special cases, the $b$-cohomology is isomorphic to Poisson cohomology: Guillemin, Miranda, and Pires \cite{Guillemin01} showed that for $b$-symplectic surfaces (see Figure \ref{fig1} for the local model) the $b$-de Rham complex computes Poisson cohomology and $b$-symplectomorphisms give all Poisson isomorphisms. However Geoffrey Scott (see \cite{Scott}, p. 18) points out that in general this is not the case. In fact, our paper is one such example of a case where the $b$-complex does not compute the entire Poisson cohomology. \vspace{3ex}  

{\bf Acknowledgment}: I am grateful to Pierre Albin for sharing many wonderful conversations over numerous cups of coffee and helping me to develop the ideas of this work. Travel support was provided by Pierre Albin's Simon's Foundation grant \#317883. I also greatly appreciate the time Ioan M\u{a}rcut spent chatting with me both in person and via emails about deformations. I also appreciate the helpful suggestions of Rui Loja Fernandes, who read a draft of the introduction. Thanks to Sarah Mousley for clarifying the topic of essential curves. Finally, I wish to thank the referee for the their 
comments and suggestions, which greatly improved the quality of the paper.

\section*{{\bf Index of Notation and Common Terms}}

\noindent \begin{tabular}{ p{3.5cm} p{9.5cm} }

$(S,D)$& A surface $S$ and a divisor $D$, that is a set of smooth hypersurfaces, i.e. curves, $Z\subset S$.\\[7pt]

$Z$ defining function \hspace{2ex}& A defining function for a hypersurface $Z\subset S$, usually denoted $x$. That is,  $x\in\mathcal{C}^{\infty}(S)$ such that 

\begin{center}$Z=\left\{{p\in S:x(p)=0}\right\}$\end{center}

and  $dx(p)\neq 0$ for all $p\in Z$. \\[7pt]

$^bTS$& b-tangent bundle over a pair $(S,D)$, the vector bundle whose sections are the vector fields on $S$ that are tangent  to $Z$ for all $Z\in D$.\\[7pt] 

$(\mathcal{A},[\cdot,\cdot]_\mathcal{A},\rho_\mathcal{A})$& A Lie algebroid over a surface $S$, that is, a triple consisting of a vector bundle $\mathcal{A}\to S$, a Lie bracket $[\cdot,\cdot]_{\mathcal{A}}$ on the $\mathcal{C}^\infty(S)$-module of sections $\Gamma (\mathcal{A})$, and a bundle map $\rho_\mathcal{A}:\mathcal{A}\to TS$ such that
 
\centerline{$\displaystyle{[X,fY]=\mathcal{L}_{\rho_{\mathcal{A}}(X)}f\cdot Y + f[X,Y]}$}\vspace{.5ex}
 
 where $X,Y \in \Gamma(\mathcal{A})$, and $f\in C^\infty(S)$.\\[7pt] 
 
 $^\mathcal{A}\Omega^{k}(S)$& The set $\Gamma(\wedge^k\mathcal{A}^*)$ of smooth sections of the $k$-th exterior power of the dual bundle to $\mathcal{A}$, called the $\mathcal{A}$-de Rham forms on $S$. This is a differential complex with differential operator defined by
 
\centerline{$\displaystyle{(d_\mathcal{A}\beta)(\alpha_0, \alpha_1,\dots, \alpha_k)=}$ } 

\centerline{$\displaystyle{\sum_{i=0}^{k}(-1)^i\rho_{\mathcal{A}}({\alpha_i})\cdot\beta(\alpha_0,\dots,\hat{\alpha}_i,\dots,\alpha_k)\hspace{1ex}+}$}

\centerline{$\displaystyle{\sum_{0\leq i<j\leq k}(-1)^{i+j}\beta([\alpha_i,\alpha_j]_\mathcal{A},\alpha_0,\dots,\hat{\alpha_i},\dots,\hat{\alpha_j},\dots,\alpha_k)}$}

\noindent for $\beta\in {^\mathcal{A}\Omega^{k}(S)}$, and  $\alpha_0,\dots, \alpha_k \in \Gamma(\mathcal{A}).$ This cohomology of this complex is called the Lie algebroid cohomology.\\[7pt]

$^0\mathcal{A}$& The vector bundle whose sections are the $\mathcal{A}$-vector fields on $S$ that are zero at the hypersurface $Z$.\end{tabular}

\section{\bf Star log symplectic surfaces} \label{section2}

In this section we construct a setting where our Poisson structures of interest can be viewed as symplectic structures. This will allow us to establish star log symplectic analogues of familiar results, such as Darboux and Moser's theorem. 

\begin{definition} A \emph{star divisor} $D$ on a surface $S$ is a finite collection $$D=\left\{c_1,c_2,\dots,c_n\right\}$$ of smooth closed curves that are pairwise transverse. 

For any $p\in S$, the \emph{degree} of $p$ is the number of curves of $D$ that contain $p$. In other words, given $p\in S$,  $$\deg(p)=\#\left\{c_i\in D: p\in c_i\right\}.$$

Note \emph{intersection points} of $D$ are precisely points $p\in D$ where $\deg(p)\geq 2$.  \end{definition} 

\subsection{Astral atlas}\label{secat} We can equip a surface with a star divisor $(S,D)$ with an atlas $\left\{(U_\alpha,\phi_\alpha)\right\}$, which we call an astral atlas, that models $D$ as the intersection of lines at the origin in the xy-plane.  

\subsubsection{Star metric} Let $(S,D)$ be a surface with a star divisor $D$. A \emph{star metric} on $S$ is a Riemannian metric $g$ such that for all intersection points $p\in S$, there exists a chart $(U,\phi)$ around $p$ such that each arc in $D\cap U$ is a geodesic.  We will provide a sketch of Max Neumann-Coto's construction of such a metric. A more detailed argument can be found in Lemma 1.2 of his paper \cite{Max}.

We start with any Riemannian metric $g_S$ on our surface $S$ and a regular neighborhood $N$ of $D$. By regular neighborhood we mean a union of neighborhoods $c_i\times(-\varepsilon,\varepsilon)$ such that the intersection of any two is either empty or diffeomorphic to $(-\varepsilon,\varepsilon)\times (-\varepsilon,\varepsilon)$. We put a flat metric $g_N$ on $N$ to make the rectangles and polygons Euclidean and so that each $c_i$ is a geodesic with respect to $g_N$. \vspace{1ex} 

\noindent \begin{minipage}{.5\linewidth}\begin{Figure} Constructing $g$. \vspace{2ex}

\includegraphics[scale=.75]{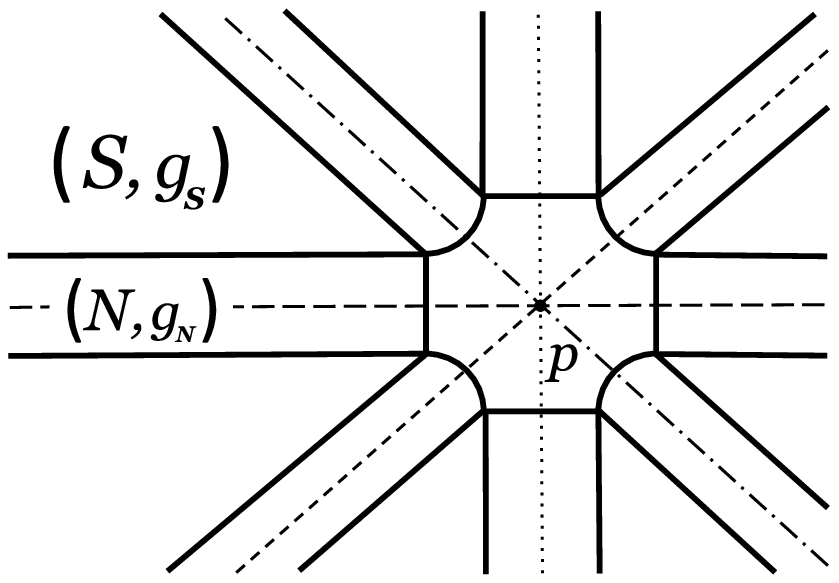} \end{Figure} \end{minipage}\hfill\begin{minipage}{.5\linewidth} 

\hspace{2ex} Let $\displaystyle{\left\{p_1,\dots,p_m\right\}}$ be the intersection points of $(S,D)$. Recall that in a closed surface $S$, an essential curve $\gamma$ is one that is not homotopic to a point, that is $[\gamma]\in\pi^1(S)$ is nontrivial. In any compact neighborhood of each $p_j$, there is a lower bound for the lengths of essential curves in $S\setminus N$. We scale $g_S$ by some positive real number $k$ so that the lengths of these essential curves with respect to $kg_S$ are greater than the lengths of each $c_i$ with respect to $g_N$.

\end{minipage}\vspace{1ex}
Using appropriate bump functions, we construct a new metric on $S$ by patching $kg_S$ and $g_N$ together. Intuitively, this new metric makes $S$ look like a mountain range around each intersection point $p_j$ with the curves in $D$ cutting out the ravines between peaks. 

\subsubsection{Star charts} A star metric allows us to construct charts that model the intersection of curves in $D$ as the intersection of lines in the plane. 
\begin{lemma}\label{starchartlem}
Let $(S,D)$  be a surface with a star divisor $D$. Let $g$ be a star metric on $(S,D)$. Fix an ordering $\left\{c_1,\dots,c_n\right\}$ of the curves in $D$. At every point $p\in S$ of degree $\geq$ 2, there exists a coordinate chart $(U,(x,y))$ centered at $p$ such that the following conditions are satisfied. 
\begin{enumerate}
\item\label{c1} Given $(\alpha,\beta)=\min \left\{(i,j)\in \mathbb{N}\times\mathbb{N}:i<j \text{ and }p\in c_i, p\in c_j\right\}$, $$c_\alpha\cap U=\left\{x=0\right\}\text{ and }c_\beta\cap U=\left\{y=0\right\}.$$
\item\label{c2} For each $c_\ell$ with $p\in c_\ell$, there exist real numbers $A_\ell$ and $B_\ell$ such that $$c_\ell\cap U=\left\{A_\ell x+B_\ell y=0\right\} .$$
\end{enumerate}
\end{lemma}

If $p\in S$ is an intersection point, then charts centered at $p$ satisfying conditions (\ref{c1}) and (\ref{c2}) from Lemma \ref{starchartlem} are called \emph{star charts}. If $p\in S$ is not an intersection point, then any chart around $p$ not containing an intersection point is called a \emph{star chart}. An \emph{astral atlas} is an atlas consisting of star charts.  

\noindent\begin{minipage}{1\linewidth} \hspace{5ex}\makebox[0pt][l]{
  \raisebox{-\totalheight}[0pt][0pt]{
    \includegraphics[scale=1 ]{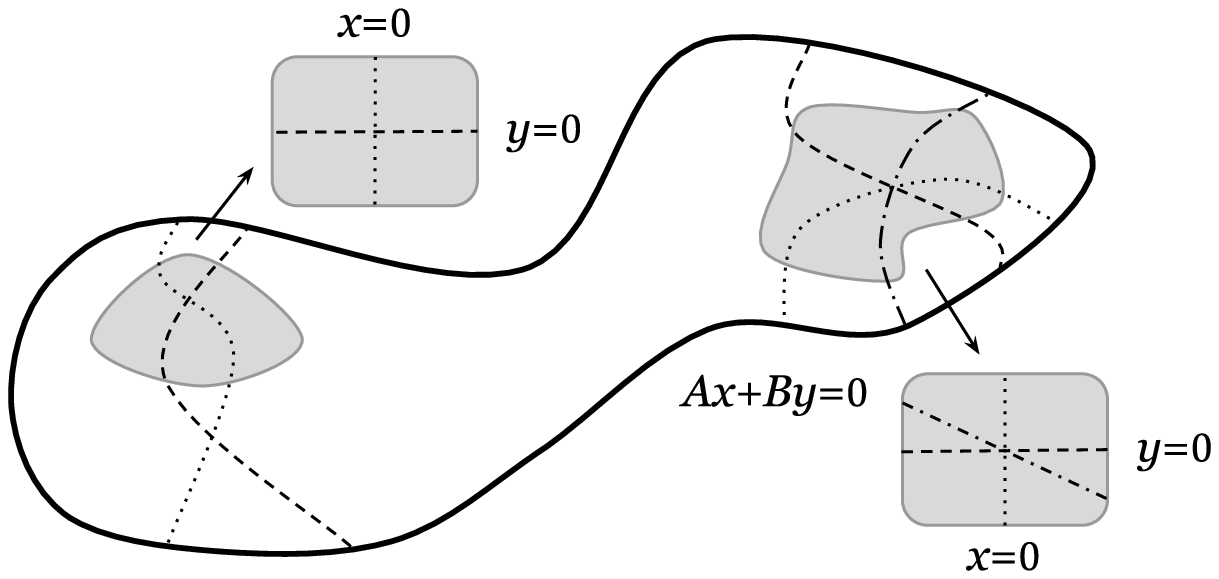}}}

\begin{minipage}{.23\linewidth}\vspace{4ex}

\begin{Figure} Star charts of $(S,D)$. \end{Figure}
\vspace{20ex}

~
\end{minipage}\hfill\begin{minipage}{.5\linewidth}
~
\end{minipage}
\end{minipage} \vspace{1ex}

\begin{proof} Let $(S,D)$  be a surface with a star divisor $D$. Consider an atlas of Riemannian normal coordinate charts. Refine this to another atlas $\left\{(U,\phi)\right\}$ by taking  charts that are also star charts.
 
Fix an ordering $\left\{c_1,\dots,c_n\right\}$ on the elements of $D$. Given an intersection point $p\in S$ and a chart $(U,\phi)$, centered at $p$, from our atlas, we will take a linear isomorphism of this chart such that the two curves passing through $p$ with the smallest indices are mapped to $\left\{x=0\right\}$ and $\left\{y=0\right\}$ respectively. An astral atlas consists of charts that have been transformed according to this ordering on $D$. \end{proof} 
  
\subsubsection{Transition functions between star charts}\label{startfcns}
Let $(S,D,g)$ be a surface, star divisor, and star metric. Given an ordering on $D$, let $\left\{(U_\alpha,\phi_\alpha)\right\}$ be an astral atlas for the triple $(S,D,g)$. Let $(U,\phi)=(U,(x,y))$ and $(V,\rho)=(V,(\tilde{x},\tilde{y}))$ be two star charts around an intersection point $p\in S$. Further assume $deg(p)\geq 3$.     

Then $\phi(U\cap D)$ is the intersection of at least 3 lines at a point. Let $c_1,c_2,c_3$ denote the curves of minimal index passing through $p$. Then $c_1$ is locally given by $\left\{x=0\right\}$ and $c_2$ is locally given by $\left\{y=0\right\}$. Further, $c_3$ is defined by an equation of the form $ax+by=0$ for non-zero real numbers $a,b$. \vspace{2ex} 
 
\noindent\begin{minipage}{1\linewidth} \makebox[0pt][l]{
  \raisebox{-\totalheight}[0pt][0pt]{
    \includegraphics[scale=.9 ]{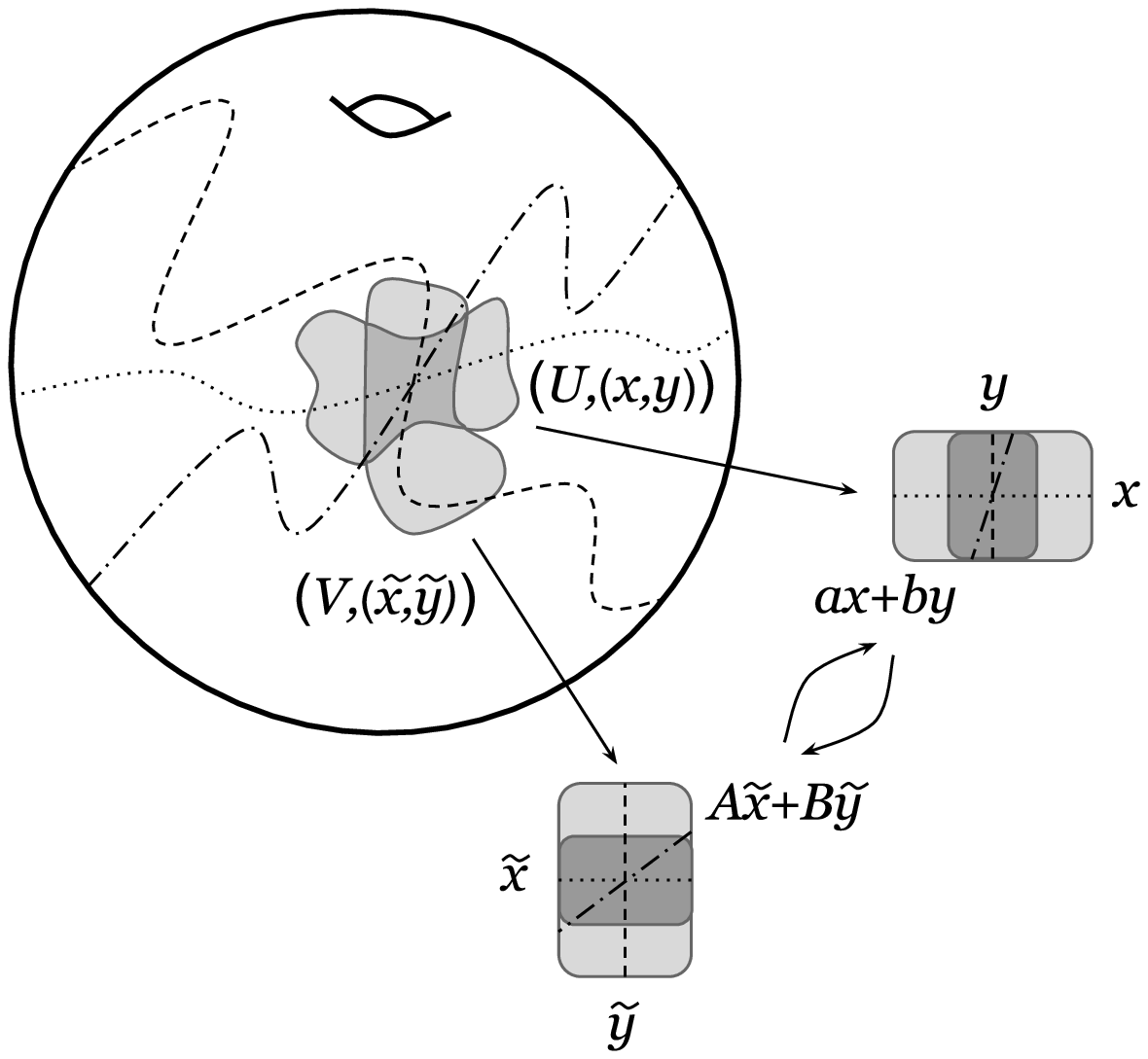}}}
    
\noindent \begin{minipage}{.5\linewidth}
 ~
\end{minipage}\hfill\begin{minipage}{.4\linewidth} Consider the transition functions between $\phi(U)$ and $\rho(V)$. In $\rho(V)$, let $\widetilde{y}$ define the image of $c_1$ and let $\widetilde{x}$ define the image of $c_2$. Then $\widetilde{x}=fx\text{ and } \widetilde{y}=gy$ for positive functions $f,g$.
\end{minipage}\vspace{24ex}
 
\noindent\begin{minipage}{.3\linewidth}
 \begin{Figure} Transition functions between star charts at a point with degree at least three.\end{Figure}
\end{minipage}\hfill\begin{minipage}{.4\linewidth} \vspace{10ex}

Then $A\widetilde{x}+B\widetilde{y}=h(ax+by)$ for $A,B\in\mathbb{R}$ and nonvanishing $h\in\mathcal{C}^\infty(U\cap V)$.  Thus we have $(Af-ha)x=(hb-Bg)y$.
\end{minipage} \end{minipage}\vspace{1ex}
 
\noindent Accordingly, there is a function $M\in\mathcal{C}^\infty(U\cap V)$ such that $$ (Af-ha)=My\text{ and }(hb-Bg)=Mx.$$ Thus $$f=(ax+by)\dfrac{M}{Ab}+\dfrac{aB}{Ab}g. $$ Notice that when $ax+by=0$, $$f=\dfrac{aB}{Ab}g.$$ Thus at the curve $c_3$, transition functions scale the $x$ and $y$ direction by the same function up to a constant. Note that off of the lines, the transition functions are less restricted - due to the function $M$ - and need only give a local diffeomorphism. 

In the case when we have at least four curves, we can extract further restrictions for transition functions. Given a chart centered at the intersection of $n\geq 4$ curves, we will show that this constant must be the same at each curve $c_3,\dots,c_n$. Consider the line $c_4$ given by $cx+ey=0=C\widetilde{x}+E\widetilde{y}$, for some $c,e,C,E\in\mathbb{R}$. Then for some smooth function $N\in\mathcal{C}^\infty(U\cap V),$ $$f=(cx+ey)\dfrac{N}{Ce}+\dfrac{cE}{Ce}g.$$ Thus \begin{equation}\label{eq:A1} (ax+by)\dfrac{M}{Ab}+\dfrac{aB}{Ab}g=(cx+ey)\dfrac{N}{Ce}+\dfrac{cE}{Ce}g \end{equation} By considering the value of equation (\ref{eq:A1}) at the origin, we conclude that $\dfrac{cE}{Ce}=\dfrac{aB}{Ab}.$ \\

%Finally, we will show that in the case of at least 4 curves intersecting at the origin, the function $M$ must vanish at the origin. We can simplify equation (\ref{eq:A1}) to the expression $$(ax+by)\dfrac{M}{Ab}=(cx+ey)\dfrac{N}{Ce}.$$ Since $ax+by\neq cx+ey$, the function $M=P(cx+ey)$ for some smooth function $P\in\mathcal{C}^\infty(U\cap V)$. Thus $M=0$ at the origin. 

In the next subsection we will use our understanding of transition functions in an astral atlas to argue that we have a certain fiber bundle structure. For a technical reason in that computation, in degree 3 charts we must impose the condition 
that our transition functions fix four lines. Below is a summary of the properties we have observed and this technical assumption. \\

\noindent {\bf Properties of transition functions between star charts at intersections of deg $\geq 3$.} \emph{Let $(S,D,g)$ be a surface, star divisor, and star metric. Given an ordering on $D$, let $\left\{(U_\alpha,\phi_\alpha)\right\}$ be an astral atlas for the triple $(S,D,g)$. Let $(U,\phi)=(U,(x,y))\text{ and }(V,\rho)=(V,(\tilde{x},\tilde{y}))$ be two star charts around an intersection point $p\in S$. Further assume $deg(p)\geq 3$.}

\emph{There are non-vanishing functions $f,g\in\mathcal{C}^\infty(U\cap V)$ so that $x$ and $y$ transition to $\widetilde{x}=fx$ and $\widetilde{y}=gy$. For every line $\ell$ in $D\cap (U\cap V)$, excluding the lines defined by $\left\{x=0\right\}$ and $\left\{y=0\right\}$, there is a real number $c$ such that $$f\big|_{\ell}=cg\big|_{\ell}.$$ }{\bf A necessary technical assumption at intersections of deg $=3$.} \emph{In charts centered at points of intersection with $deg=3$, there exists a line through the origin $L\not\in D\cap(U\cap V)$  such that $f|_L=cg|_L$. }

\newpage 

\subsection{b-tangent bundle over an astral atlas} 
Given a surface $S$ with star divisor $D$, we will use an astral atlas to construct a vector bundle whose smooth sections are vector fields tangent to $D$. We will begin by defining a vector bundle over each chart. Then, by arguing that all the transition maps are compatible with this structure, we will show that we have in fact constructed a rank 2 vector bundle. 
Let $p_1,\dots,p_m$ be all the points of $S$ with degree greater than $2$. Consider the surface $\widetilde{S}=S\setminus\left\{p_1,\dots,p_m\right\}$. Then  $$(\widetilde{S},\widetilde{D})=(S\setminus\left\{p_1,\dots,p_m\right\},D\cap \widetilde{S})$$ is a surface with normal crossing divisor, that is, a surface with a set of smooth curves that intersect transversely. The b-tangent bundle ${}^bT\widetilde{S}$, referred to as the log tangent bundle ${}^{log}T\widetilde{S}$ in \cite{Gualtieri02}, is the vector bundle whose smooth sections are the vector fields of $\widetilde{S}$ that are tangent to $\widetilde{D}$. In other words, the b-tangent bundle has smooth sections 

$$\left\{u\in\mathcal{C}^\infty(\widetilde{S},T\widetilde{S}): u|_c\in \mathcal{C}^\infty(c,Tc) \text{ for all } c\in \widetilde{D}\right\}.$$

This will be our vector bundle away from $p_1,\dots,p_m$, the intersection points of degree greater than two. Let $\left\{(U_\alpha,\phi_\alpha)\right\}$ be an astral atlas of $(S,D)$. 

Consider a degree $k$ point $p\in S$ where $k\geq 3$. Then $p$ sits in some curves $c_1,\dots,c_k$. Let $(U,(x,y))$ be a star chart around $p$. Consider the vector fields, depicted in Figure \ref{vfgraph}, $$V=x\dfrac{\partial}{\partial x}+y\dfrac{\partial}{\partial y}\text{ and }W=x\dfrac{\partial}{\partial x}-y\dfrac{\partial}{\partial y}.$$ 
\begin{Figure}\label{vfgraph} Vector fields on $TU$. \vspace{2ex}

\begin{minipage}{.5\textwidth}\begin{center} 
 \includegraphics[scale=.5]{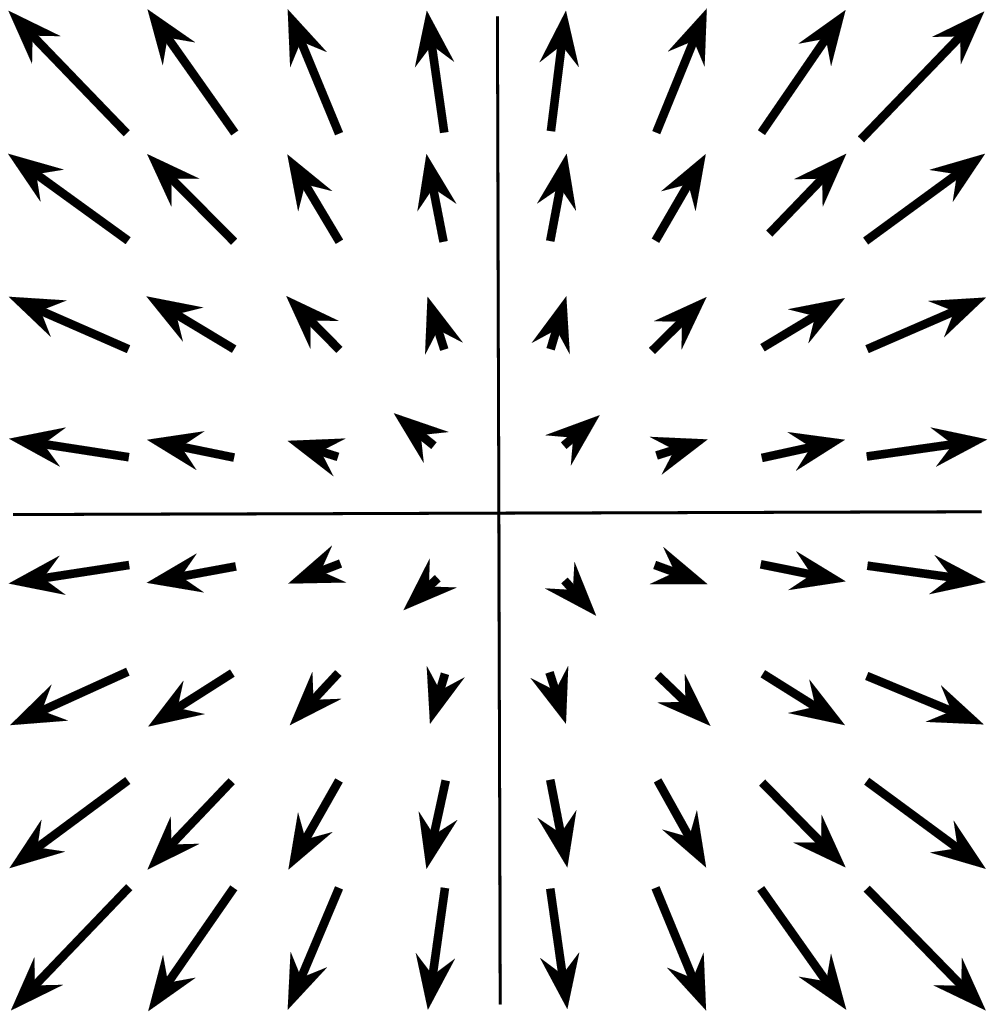}
 
 {\bf $V$} 
 
 \end{center}
 \end{minipage} \begin{minipage}{.5\textwidth}\begin{center}\includegraphics[scale=.5]{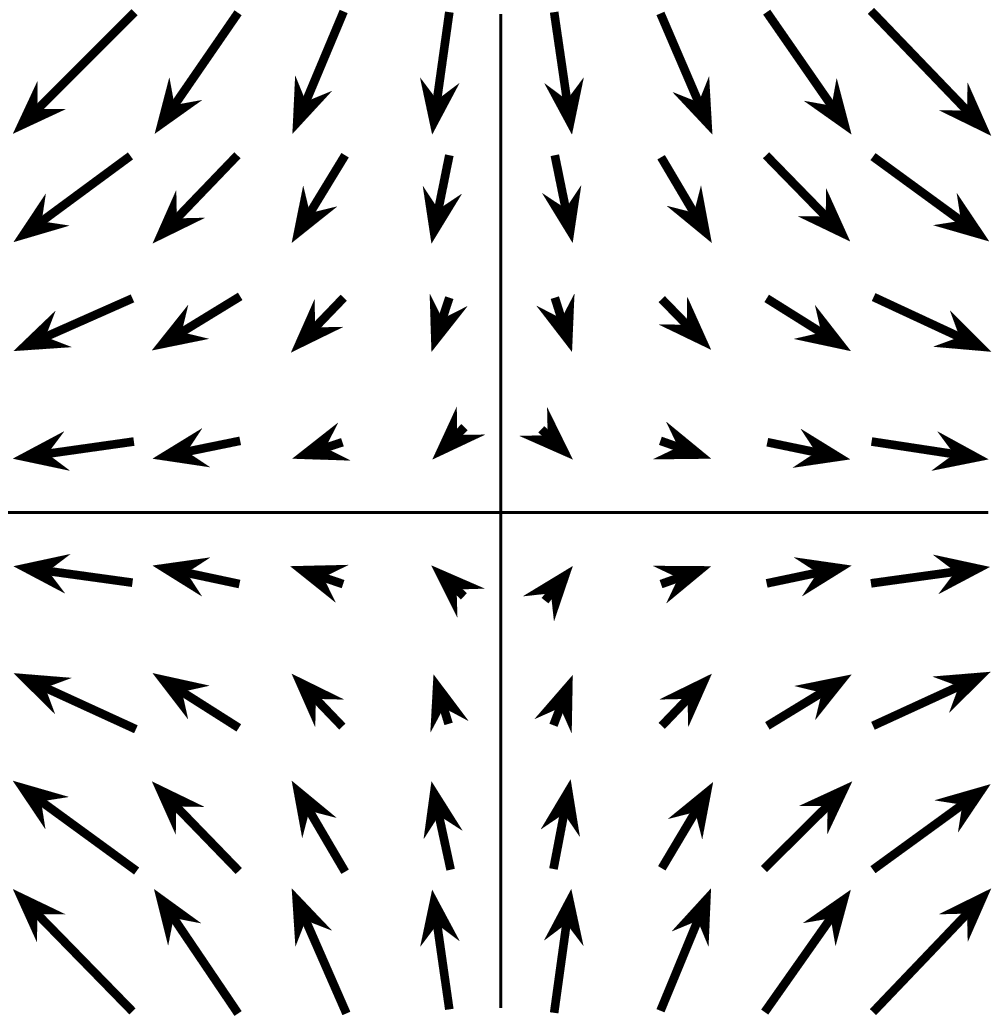}
 
 {\bf $W$} 
 \end{center}
 \end{minipage} 
 
\end{Figure}

Away from $c_1$ and $c_2$, these generate $TU$. Further, away from the origin, $V$ is tangent to all lines passing through the origin. \vspace{2ex}

Let $c_3,\dots,c_k$ be locally defined by $A_3x+B_3y,\dots$, and $A_kx+B_ky$ respectively, i.e. $c_i\cap U=\left\{A_ix+B_iy=0\right\}$. Then we define two vector fields \begin{equation}\label{eqvf}
V=x\dfrac{\partial}{\partial x}+y\dfrac{\partial}{\partial y}\hspace{2ex}\text{ and }\hspace{2ex}\widetilde{W}=\prod_{i=3}^k\Big(A_ix+B_iy\Big)\left(x\dfrac{\partial}{\partial x}-y\dfrac{\partial}{\partial y}\right).\end{equation}

In the following lemma we will show that $V$ and $\widetilde{W}$ provide a basis for the space of smooth vector fields tangent to $D\cap U$. 

\begin{lemma} Let $(U,(x,y))$ be a star chart around a point $p\in S$ of degree at least 3. Any vector field $V\in\mathcal{C}^\infty(U;TU)$ that is tangent to $D\cap U$ can be expressed as a $\mathcal{C}^\infty(U)$-linear combination of the vector fields given in equation (\ref{eqvf}). 
\end{lemma}

\begin{proof} 
Given a star chart $(U,(x,y))$ centered at $p$, by definition the set $D\cap U$ is given by the lines $x=0, y=0$,  $A_3x+B_3y=0$,\ldots, $A_kx+B_ky=0$. We want to find all vector fields that are tangent to these lines. 

An arbitrary smooth vector field in $U$ has the form $$V=a(x,y)\dfrac{\partial}{\partial x}+b(x,y)\dfrac{\partial}{\partial y}$$ for $a(x,y), b(x,y)\in\mathcal{C}^\infty(U)$. \vspace{2ex}

{\bf Claim:} The functions $a$ and $b$ satisfy $$ a=x\left(\dfrac{hy\prod_{i=4}^k\Big(A_ix+B_iy\Big)}{A_3}+c_3\right)\text{ and }
b=y\left(c_3-\dfrac{hx\prod_{i=4}^k\Big(A_ix+B_iy\Big)}{B_3}\right) 
$$ for arbitrary smooth functions $h,c_3\in\mathcal{C}^\infty(U)$. \vspace{2ex}

We will prove this claim by induction on the degree of $p$. \vspace{2ex}

{\bf Base case:} Assume $\deg p =3$. Notice that $V$ is tangent to the line $x=0$ if and only if $$\left<V\cdot \dfrac{\partial}{\partial x}\right>_{g_{\mathbb{R}^2}}\bigg|_{\left\{x=0\right\}}=0.$$ Equivalently, $V$ is tangent to $x=0$ if and only if $a(x,y)|_{\left\{x=0\right\}}=0$, i.e. \begin{equation}\label{e1} a(x,y)=a^\prime(x,y)x\end{equation} for $a^\prime(x,y)\in\mathcal{C}^\infty(U)$.  Similarly, $V$ is tangent to the line $y=0$ if and only if $b(x,y)|_{\left\{y=0\right\}}=0$. That is, \begin{equation}\label{e2} b(x,y)=b^\prime(x,y)y\end{equation} for $b^\prime(x,y)\in\mathcal{C}^\infty(U)$. Because the gradient of a function is normal to its level sets, $V$ is tangent to $A_3x+B_3y=0$ precisely when $$\left<V\cdot\left(A_3\dfrac{\partial}{\partial x}+B_3\dfrac{\partial}{\partial y}\right)\right>_{g_{\mathbb{R}^2}}\bigg|_{\left\{A_3x+B_3y=0\right\}}=0.$$ Equivalently, $V$ is tangent to $A_3x+B_3y=0$ if and only if \begin{equation}\label{e3} A_3a(x,y)+B_3b(x,y)=c_3(x,y)(A_3x+B_3y)\end{equation} for $c_3(x,y)\in\mathcal{C}^\infty(U)$. 

We want to find all possible $a,b$ satisfying equations (\ref{e1}), (\ref{e2}), and (\ref{e3}). Substituting $a=a^\prime x$ and $b=b^\prime y$, we have $$A_3a^\prime x+B_3b^\prime y=c_3(A_3x+B_3y).$$ Rearranging terms, $$xA_3(a^\prime -c_3)=yB_3(c_3-b^\prime).$$ Notice that $y$ divides both the right and left hand sides, and thus $$A_3(a^\prime-c_3)=ey$$ for some $e\in\mathcal{C}^\infty(U)$. Further, $B_3(c_3-b^\prime)=ex$. We can solve for $a^\prime$ and $b^\prime$ respectively: $$ a^\prime= \dfrac{ey}{A_3}+c_3\text{ and }
b^\prime= c_3-\dfrac{ex}{B_3}.$$ Thus, we have shown that if $V$ is tangent to the lines $x=0,y=0$, and $A_3x+B_3y=0$, then $$ a=x\left( \dfrac{ey}{A_3}+c_3\right)\text{ and }
b=y\left(c_3-\dfrac{ex}{B_3}\right).$$ 

Note that for any smooth functions $c_3,e\in \mathcal{C}^\infty(U)$, we have 

\centerline{$a=\mathcal{O}(x)$, $b=\mathcal{O}(y)$, and $A_3a+B_3b=\mathcal{O}(A_3x+B_3y)$.}

Consequently, there are no restrictions on $c_3$ or $e$ and we have a characterization of all vector fields that are tangent to these three lines.  \vspace{2ex}

{\bf Induction step:} Assume that the claim is true for a point $p$ satisfying $\deg p=n-1$. We will verify the formula at a point $p$ of degree $n$. \vspace{2ex}

The set $D\cap U$ is given by the lines $$x=0, y=0,A_3x+B_3y=0,\ldots,A_nx+B_ny=0.$$ We want to find all vector fields that are tangent to these lines. Since a vector field tangent to these lines is certainly also tangent to the lines $$x=0, y=0,  A_3x+B_3y=0,\ldots, A_{n-1}x+B_{n-1}y=0,$$ then we know that $a$ and $b$ at least satisfy 
$$ a=x\left(\dfrac{hy\prod_{i=4}^{n-1}\Big(A_ix+B_iy\Big)}{A_3}+c_3\right),
b=y\left(c_3-\dfrac{hx\prod_{i=4}^{n-1}\Big(A_ix+B_iy\Big)}{B_3}\right). 
$$

Notice that $V$ is tangent to the line $A_nx+B_ny=0$ precisely when $$\left<V\cdot\left(A_n\dfrac{\partial}{\partial x}+B_n\dfrac{\partial}{\partial y}\right)\right>_{g_{\mathbb{R}^2}}\bigg|_{\left\{A_nx+B_ny=0\right\}}=0.$$ Equivalently, $V$ is tangent to $A_nx+B_ny=0$ if and only if \begin{equation}\label{e4} A_na(x,y)+B_nb(x,y)=c_n(x,y)(A_nx+B_ny)\end{equation} for $c_n(x,y)\in\mathcal{C}^\infty(U)$. 

By substituting in our expressions for $a$ and $b$ and rearranging terms, we have  $$c_3(A_nx+B_ny)+hxy\prod_{i=4}^{n-1}\left(A_ix+B_iy\right)\left(\dfrac{A_n}{A_3}-\dfrac{B_n}{B_3}\right)=c_n(A_nx+B_ny).$$

Note that $\frac{A_n}{A_3}-\frac{B_n}{B_3}=0$ if and only if $A_nx+B_ny=0$ defines the same line as $A_3x+B_3y=0$. Notice that $A_nx+B_ny$ divides both the right hand side and the term $c_3(A_nx+B_ny)$, and thus $h=k(A_nx+B_ny)$ for some smooth function $k\in\mathcal{C}^\infty(U)$. 

Thus, we have shown if $V$ is tangent to the lines $$x=0, y=0,A_3x+B_3y=0,\ldots, A_nx+B_ny=0,$$ then $$ a=x\left(\dfrac{ky\prod_{i=4}^{n}\Big(A_ix+B_iy\Big)}{A_3}+c_3\right),
b=y\left(c_3-\dfrac{kx\prod_{i=4}^{n}\Big(A_ix+B_iy\Big)}{B_3}\right). 
$$ Note that for any smooth functions $c_3,k\in \mathcal{C}^\infty(U)$, we have 

\centerline{$a=\mathcal{O}(x)$, $b=\mathcal{O}(y)$, \ldots, and  $A_na+B_nb=\mathcal{O}(A_nx+B_ny)$.} Consequently, there are no restrictions on $c_3$ or $k$ and we have a characterization of all vector fields that are tangent to a set of $n$ lines passing through the origin. By choosing $c_3=1,k=0$ and  $\displaystyle{c_3=(A_3x-B_3y)\prod_{i=4}^{n}(A_ix+B_iy)}$, $k=2A_3B_3$, we have a local frame provided by the vector fields defined in equation \ref{eqvf}. 

\end{proof}

Our final task to construct the $b$-tangent bundle over a star divisor is to show that the local trivializations $V$ and $\widetilde{W}$ are preserved under the transition functions of an astral atlas. 

\begin{lemma} Let $(S,D,g)$ be a surface, star divisor, and star metric with an associated astral atlas $\left\{(U_\alpha,\phi_\alpha)\right\}$. The local vector fields defined in equation (\ref{eqvf}) and transition functions given in Section \ref{startfcns} define a vector bundle over $S$.\end{lemma} 

\begin{proof} Assume we are considering a chart $U$ centered at a point $p$ where at least three curves intersect. Let $x$, $y$, and $ax+by$ define  the first three of these curves. As discussed in Section \ref{startfcns}, when transitioning from $x$ and $y$ to a chart $T$ with $\widetilde{x}$ and $\widetilde{y}$, we have the following relations: $\widetilde{x}=fx$, $\widetilde{y}=gy$, and $A\widetilde{x}+B\widetilde{y}=h(ax+by)$ for some nowhere vanishing functions $f,g,h\in\mathcal{C}^{\infty}(U\cap T)$.  Further, $$f=(ax+by)\dfrac{M}{Ab}+\dfrac{aB}{Ab}g$$ for some smooth function $M\in\mathcal{C}^\infty(U\cap T)$. 

To check that $V$ and $\widetilde{W}$ are preserved under our transition functions, it is enough to show that the corresponding co-vectors $V^\ast$ and $\widetilde{W}^\ast$ are preserved using the same open cover and transition functions. That is, we need to check that we still have generators of the form $$V^\ast=\dfrac{dx}{ x}+\dfrac{dy}{y}\hspace{2ex}\text{ and }\hspace{2ex}\widetilde{W}^\ast=\prod_{i=3}^k\Big(A_ix+B_iy\Big)^{-1}\left(\dfrac{dx}{x}-\dfrac{dy}{y}\right)$$ when we transition from $x$, $y$ to $\widetilde{x}, \widetilde{y}$. 

We will first consider $V^\ast$. Note that $$\dfrac{d\widetilde{x}}{\widetilde{x}}+\dfrac{d\widetilde{y}}{\widetilde{y}}=\dfrac{dx}{x}+\dfrac{dy}{y}+\dfrac{df}{f}+\dfrac{dg}{g}.$$ Since $f$ and $g$ are nowhere vanishing functions, the fiber generated by $V^\ast$ expressed in $x$, $y$ maps to the fiber generated by $V^\ast$ expressed in $\widetilde{x}$, $\widetilde{y}$.  

We are left to consider $\widetilde{W}^\ast$. Note that $$\prod_{i=3}^k\Big(A_i\widetilde{x}+B_i\widetilde{y}\Big)^{-1}\left(\dfrac{d\widetilde{x}}{\widetilde{x}}-\dfrac{d\widetilde{y}}{\widetilde{y}}\right)=\prod_{i=3}^k\Big(A_ifx+B_igy\Big)^{-1}\left(\dfrac{dx}{x}-\dfrac{dy}{y}+\dfrac{df}{f}-\dfrac{dg}{g}\right).$$

Recall from Section \ref{startfcns} that at the origin $f=Cg$ for some constant $C$. Thus $$\widetilde{W}^\ast\bigg|_{(0,0)}=\left(\dfrac{1}{Cg}\right)^{k-3+1}\prod_{i=3}^k\Big(A_ix+B_iy\Big)^{-1}\left(\dfrac{dx}{x}-\dfrac{dy}{y}+\dfrac{df}{f}-\dfrac{dg}{g}\right)\bigg|_{(0,0)}.$$ 

Since $g$ is nowhere vanishing, if we can show that $\dfrac{df}{f}-\dfrac{dg}{g}$ vanishes at the origin, we have shown the fiber generated by $\widetilde{W}^\ast$ expressed in $x,y$ is preserved under transition to the fiber generated by $\widetilde{W}^\ast$ expressed in $\widetilde{x},\widetilde{y}$. 

Let $\ell$ and $L$ be two distinct lines in $U\cap V$, intersecting at the origin, such that $f|_\ell=cg|_\ell$ and $f|_L=cg|_L$ for a real number $c$. The one-form $\dfrac{df}{f}-\dfrac{dg}{g}$ is zero when evaluated on $\ell$ and when evaluated on $L$; since $\ell$ and $L$ span the dual space at the origin, the one form must vanish at the origin.

To conclude, because we have local trivializations of our fibers with a suitable set of transition functions, by the fiber bundle construction theorem we have constructed a vector bundle over $(S, D)$.\end{proof} 

This vector bundle comes equipped with a Lie algebroid structure. The anchor map is evaluation in $TS$ and the Lie bracket is induced by the standard Lie bracket on $TS$. 

Let us consider an explicit example of a $b$-tangent bundle over a star divisor.

\begin{Example}\label{exb} Let $\mathbb{R}^2$ be equipped with the standard Euclidean coordinates \vspace{.4ex}

\noindent \begin{minipage}{.6\linewidth}and let star divisor $D$ consist of the lines $$\left\{x=0\right\},\left\{y=0\right\}, \left\{x+y=0\right\}, \text{ and }\left\{x-y=0\right\}.$$ The $b$-tangent bundle is generated by the vector fields $$x\dfrac{\partial}{\partial x}+y\dfrac{\partial}{\partial y}, (x-y)(x+y)\left(x\dfrac{\partial}{\partial x}-y\dfrac{\partial}{\partial y}\right) $$ and the $b$-cotangent bundle is generated by  $$\dfrac{dx}{x}+\dfrac{dy}{y}, \dfrac{1}{(x-y)(x+y)}\left(\dfrac{dx}{x}-\dfrac{dy}{y}\right).$$

\end{minipage} \begin{minipage}{.4\linewidth} \begin{center}\vspace{-2ex}

\hspace{3ex}\includegraphics[scale=1]{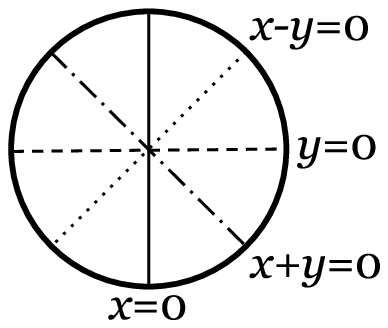}

{\bf divisor $D$}\end{center} \end{minipage}  
\end{Example}\vspace{1ex}

\begin{definition} A \emph{star log symplectic structure} on a surface $S$ with star divisor $D$ is a 2-form $\omega\in{}^{b}\Omega^2(S)=\mathcal{C}^\infty\big(S;\wedge^2\big({}^{b}T^*S\big)\big)$ satisfying $$d\omega=0\text{ and }\omega^n\neq 0.$$ The form $\omega$ induces a map $\omega^\flat$ between the $b$-tangent and $b$-cotangent bundles. \centerline{$\xygraph{!{<0cm,0cm>;<1cm, 0cm>:<0cm,1cm>::}
!{(1.25,0)}*+{{}^{b}TS}="b"
!{(4.9,0)}*+{{{}^{b}T^*S}}="c"
!{(1.75,.1)}*+{{}}="d"
!{(4.25,.1)}*+{{}}="e"
!{(1.75,-.1)}*+{{}}="f"
!{(4.25,-.1)}*+{{}}="g"
"d":"e"^{\omega^{\flat}}
"g":"f"^{\pi^{\sharp}=(\omega^{\flat})^{-1}}}$}

\noindent The inverse map is induced by a bi-vector $\pi\in\mathcal{C}^\infty(S;\wedge^2({}^{b}TS))$.  This bi-vector is called a \emph{log Poisson structure} on $(S,D)$.\end{definition} 

%\begin{Example} Continuing with the $b$-tangent bundle from Example \ref{exb}, the form $$\omega =  \dfrac{1}{2(x-y)(x+y)}\left(\dfrac{dx}{x}-\dfrac{dy}{y}\right)\wedge\left(\dfrac{dx}{x}+\dfrac{dy}{y}\right)= \dfrac{dx\wedge dy}{(x+y)(x-y)xy} $$ is a log-symplectic structure on $\mathbb{R}^2$. The associated log Poisson bi-vector is given by $$\pi=(x+y)(x-y)xy\dfrac{\partial}{\partial y}\wedge \dfrac{\partial}{\partial x}.$$\end{Example}
 
Next, we provide an explicit example of a star log symplectic structure on a compact surface. 

%\begin{Example} {\bf A quadratic log symplectic surface} 

%\vspace{-2ex} \begin{multicols}{2} 
%\begin{center}  $~$

%$(\mathbb{T}^2,\cup_{2}(\mathbb{S}^1\cup\mathbb{S}^1))$\vspace{2ex}

%\includegraphics[scale=1]{ClassPic03.eps}\end{center}\columnbreak 

%\noindent Consider the torus $\mathbb{T}^2$. Let $a_1,a_2$ be the angular coordinates on $\mathbb{T}^2$. We have a divisor $D$ consisting of the 2 hypersurfaces $$\left\{\sin(a_1)=0\right\}$$

%$$\text{ and }\left\{\sin(a_2)=0\right\}.$$ The form $$\omega = \dfrac{da_1\wedge da_2}{\sin(a_1)\sin(a_2)}$$  is a log symplectic structure on $\mathbb{T}^2$. 

%\end{multicols}
%\vspace{-2ex}\noindent The associated Poisson bivector is given by $$\pi =\sin(a_1)\sin(a_2)\dfrac{\partial}{\partial a_2}\wedge\dfrac{\partial}{\partial a_1}.$$

%\end{Example}

\newpage 
\begin{Example}\label{motivatingexample} {\bf A cubic star log symplectic surface.}

\vspace{-2ex}\begin{multicols}{2} \begin{center} $~$

  $(\mathbb{S}^2,\cup_{3}\mathbb{S}^1)$\vspace{2ex} 

\includegraphics[scale=1]{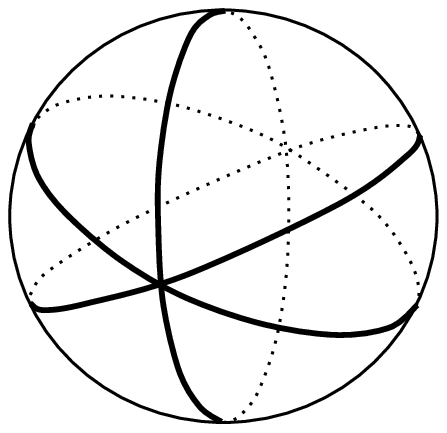}\end{center}

\columnbreak \noindent Consider the sphere  $\mathbb{S}^2$ with spherical coordinates $0\leq \theta < 2\pi$, $0\leq h\leq \pi$. We have a star divisor $D$ consisting of  $$\left\{\sin(\theta-\pi/6)=0\right\},$$
 $$\left\{\sin\left(\theta-\pi/2\right)=0\right\},\text{ and }$$  $$\left\{\sin\left(\theta-5\pi/6\right)=0\right\}.$$ The form $$ \omega = \dfrac{d\theta\wedge dh}{\sin(\theta-\frac{\pi}{6})\sin(\theta-\frac{\pi}{2})\sin(\theta-\frac{5\pi}{6})}$$  is a star log symplectic structure on $\mathbb{S}^2$. 
\end{multicols} \vspace{-2ex} \noindent The associated Poisson bivector is given by $$\pi =\sin(\theta-\frac{\pi}{6})\sin(\theta-\frac{\pi}{2})\sin(\theta-\frac{5\pi}{6}) \dfrac{\partial}{\partial h}\wedge\dfrac{\partial}{\partial \theta}.$$ \end{Example}

 \section{\bf Classification}\label{SectionClassification}  The first step in classifying star log symplectic surfaces is computing the Lie algebroid cohomology of the $b$-tangent bundle. 

\subsection{\bf $b$-de Rham cohomology} 

Given a surface $S$ with a collection of transverse curves $\left\{Z_1,\dots,Z_k\right\}$, the de Rham cohomology of the associated $b$-tangent bundle is  \begin{equation}\label{eq:1}{}^bH^p(S)\simeq H^p(S)\oplus \bigoplus_{i}H^{p-1}(Z_i)\oplus\bigoplus_{i<j}H^{p-2}(Z_i\cap Z_j),\end{equation} a fact originally pointed out in Appendix A.23 of \cite{Gualtieri02}. Remarkably, this decomposition still holds for the more general case of a surface with star divisor!  

\begin{theorem} \label{bstarcohom}
Let $(S,D)$ be a surface $S$ with star divisor $D=\left\{Z_1,\dots,Z_k\right\}$. The Lie algebroid cohomology of the b-tangent bundle over $(S,D)$ is given by (\ref{eq:1}). 
\end{theorem} 

\begin{proof}

We will prove Theorem \ref{bstarcohom} using a rather simple spectral sequence, see Chapter 1 of \cite{McCleary} for a friendly introduction. \vspace{1ex}

\noindent {\bf Constructing a filtration $\mathcal{F}$.} 

Let $(S,D)$ be a surface $S$ with star divisor $D$. Given a subset $I\subseteq D$, let ${}^b\Omega^\ast(S,I)$ denote the complex of Lie algebroid forms of the b-tangent bundle over $(S,I)$. In this notation, we want to compute the cohomology of the complex ${}^b\Omega^\ast(S,D)$. 

We will filter the differential complex $({}^b\Omega^\ast(S),d)$ using a filtration denoted $$\mathcal{F}=\left\{\Omega^\ast(S)=\mathscr{F}^\ast_0\subseteq \mathscr{F}^\ast_1\subseteq \mathscr{F}^\ast_2\subseteq\dots\subseteq \mathscr{F}^\ast_{|D|-1}\subseteq \mathscr{F}^\ast_{|D|}={}^b\Omega^\ast(S,D)\right\}$$
 where $$\mathscr{F}^*_k:={}^{v.s.}\text{Span}\left\{\bigcup_{\text{\scalebox{.8}{$I\subseteq D,|I|=k$}}}\hspace{-1ex}{}^b\Omega^\ast(S,I)\right\}.$$ Notice that $\mathscr{F}^\ast_k\subseteq \mathscr{F}^\ast_{k+1}$ for all $k$. Each set $\mathscr{F}^p_k$ is a group under addition and has the structure of a module with scalars from the ring $$\mathscr{F}^{0}_k=\mathcal{C}^\infty(S).$$ 

Note that ${}^b\Omega^\ast(S,I)\subseteq {}^b\Omega^\ast(S,D)$ for all $I\subseteq D$. Consequently, the modules $\mathscr{F}_k^\ast\subseteq {}^b\Omega^\ast(S,D)$ for all $k$. It is easy to check that $\mathscr{F}_k^\ast$ is closed under the differential of ${}^b\Omega^\ast(S,D)$. Thus we can define a differential $\mathscr{F}_k^p\xrightarrow{~d~}\mathscr{F}_k^{p+1}$ inherited from the differential on ${}^b\Omega^\ast(S,D)$. 

For each $k$, we have a chain complex: $$0\to \mathscr{F}^0_k\xrightarrow{~d~}\mathscr{F}^1_k\xrightarrow{~d~}\mathscr{F}^2_k\to 0.$$    \vspace{1ex}

\noindent {\bf The associated graded complex of $\mathcal{F}$.}

Next we will consider the associated graded complex of our filtered complex $({}^b\Omega^\ast(S),d, \mathcal{F})$. We denote the associated graded module $\mathscr{F}^p_k/\mathscr{F}^p_{k-1}$ by $E^{p,k}_0$. 

The inclusions $\mathscr{F}^\ast_{k-1}\to \mathscr{F}^\ast_{k}$ of cochain complexes from filtration $\mathcal{F}$ fit into a series of short exact sequence of complexes \begin{center} $0\to \mathscr{F}^\ast_{k-1}\to \mathscr{F}^\ast_{k}\to E^{\ast,k}_0\to 0$. \end{center} The differential $^{E}d$ on $E^{p,k}_0$ is induced by $d$ on $\mathscr{F}^p_{k}$: if $P$ is the projection $\mathscr{F}^p_k\to\mathscr{F}^p_k/\mathscr{F}^p_{k-1},$ then  $^{E}d(\eta)=P(d(\theta))$ where  $\theta\in\mathscr{F}^p_k$ is any form such that $P(\theta)=\eta$. Hence $({}^Ed)^2=0$ and $(E^{\ast,k}_0,{}^Ed)$ is a complex. \vspace{2ex}

The associated graded complex of $({}^b\Omega^\ast(S),d)$ is $$\bigoplus_{k=1}^m E^{\ast,k}_0.$$ The following computations will show that the spectral sequence collapses at the first page and $$ \bigoplus_{k=1}^m H^d(E^{\ast,k}_0)\simeq{}^bH^d(S), $$ or that the cohomology of the associated graded complex computes the cohomology of $({}^b\Omega^\ast(S),d)$. \vspace{2ex}

\noindent {\bf Computing $H^p(\mathscr{F}^\ast_1)$.}

Let $k=1$ and consider the short exact sequence $$0\to \mathscr{F}^\ast_{0}\to \mathscr{F}^\ast_{1}\to E^{\ast,1}_0\to 0$$ defined above. Note that $\mathscr{F}^\ast_0=\Omega^\ast(S)$. 

Let $D=\left\{Z_1,Z_2,\dots,Z_\ell\right\}$ be our star divisor and consider defining functions $x_1,x_2,\dots,x_\ell$ such that $Z_i=\left\{x_i=0\right\}$. For each $i$, let $\chi_{i}:S\to\mathbb{R}$ be a smooth bump function  supported in a neighborhood of $Z_i$ and with constant value one at $Z_i$. 

An element $\mu \in \mathscr{F}_1^p$ can be expressed as $$\mu=\sum_{i=1}^\ell \chi_i\cdot \dfrac{dx_i}{x_i}\wedge\alpha_i+\beta$$ where $\alpha_i\in\Omega^{p-1}(Z_i)$ and $\beta\in\Omega^p(S)$.

Given $\mu\in \mathscr{F}_1^p$, we write $\mathfrak{R}(\mu)=\beta$ and $\mathfrak{S}(\mu)=\mu-\mathfrak{R}(\mu)$ for `regular' and `singular' parts. It is easy to see that $\mathfrak{R}(d\mu)=d(\mathfrak{R}(\mu))$ and $\mathfrak{S}(d\mu)=d(\mathfrak{S}(\mu))$. Thus we have a  splitting $$\mathscr{F}^\ast_1=\Omega^*(S)\oplus E^{\ast,1}_0$$ as complexes and $H^p(\mathscr{F}^\ast_1)\simeq H^p(S)\oplus  H^p(E^{\ast,1}_0)$. We are left to compute the cohomology of the complex $E^{\ast,1}_0$. 

As mentioned above, $P(\mu)\in E^{\ast,1}_0$ is of the form $$P(\mu)=\sum_{i=1}^\ell \dfrac{dx_i}{x_i}\wedge\alpha_i$$ for $\alpha_i\in\Omega^{p-1}(Z_i)$. Then $$dP(\mu)=-\sum_{i=1}^\ell \dfrac{dx_i}{x_i}\wedge d\alpha_i.$$ 

Note that as an element of $\mathscr{F}^\ast_1$, $$\mu|_{Z_i}=\dfrac{dx_i}{x_i}\wedge \alpha_i.$$ Thus, the condition that $\mu|_{Z_i}=0$ implies that $\alpha_i=0$. Consequently, given these choices of $Z_i$ defining functions $x_1,\dots,x_\ell$, the set $\ker (d:E^{p,1}_0\to E^{p+1,1}_0)$ can be identified with $$\bigoplus_{i=1}^\ell \left\{\alpha_i\in\Omega^{p-1}(Z_i):d\alpha_i=0\right\}.$$ Further, the set $\text{im} (d:E^{p-1,1}_0\to E^{p,1}_0)$ can be identified with $$\bigoplus_{i=1}^\ell \left\{\alpha_i\in\Omega^{p-1}(Z_i):\alpha_i=d\gamma_i, \gamma_i\in\Omega^{p-2}(Z_i)\right\}.$$ 

For completeness, we next consider change of $Z_i$ defining function. However, by a computation which can be found in example 2.11 of \cite{Lanius}, these sets are invariant under change of $Z_i$ defining functions. Thus $$H^p(E^{\ast,1}_0)\simeq \bigoplus_i H^{p-1}(Z_i)\text{ and }H^p(\mathscr{F}^\ast_1)\simeq H^p(S)\oplus\bigoplus_i H^{p-1}(Z_i).$$

\vspace{2ex}

\noindent {\bf Computing $H^p(\mathscr{F}^\ast_2)$.} 

Let $k=2$ and consider the short exact sequence $$0\to \mathscr{F}^\ast_{1}\to \mathscr{F}^\ast_{2}\to E^{\ast,2}_0\to 0.$$

Let $D=\left\{Z_1,Z_2,\dots,Z_\ell\right\}$ be our star divisor. For each pair of curves $Z_i, Z_j$ with $i<j$, there is a pair of tubular neighborhoods $\tau_i=Z_i\times(-\varepsilon, \varepsilon)_{x_{ij}}$ of $Z_i$ and $\tau_j=Z_j\times(-\varepsilon, \varepsilon)_{y_{ij}}$ of $Z_j$ such that $$\left[\dfrac{\partial}{\partial x_{ij}},\dfrac{\partial}{\partial y_{ij}}\right]=0.$$ For existence of such neighborhoods see, for instance, section 5 of \cite{Albin}.  

We begin by noting that $\mathscr{F}_2^1=\mathscr{F}_1^1$.

Let $\chi_{i,j}$ be a smooth bump function supported near $Z_i\cup Z_j$ and with constant value one near $Z_i\cup Z_j$. An element $\mu \in \mathscr{F}_2^2$ can be expressed as $$\mu=\sum_{i<j}\chi_{ij}\cdot \dfrac{dx_{ij}}{x_{ij}}\wedge\dfrac{dy_{ij}}{y_{ij}} \alpha_{ij}+\beta$$ where $\alpha_{ij}\in\Omega^{0}(Z_i\cap Z_j)$ and $\beta\in\mathscr{F}_1^2$.

Because $\mathscr{F}_1^\ast$ is closed under $d$ and $d\mu=0$ for all $\mu\in\mathscr{F}_2^2$, we have a  splitting $$\mathscr{F}^\ast_2=\mathscr{F}^\ast_1\oplus E^{\ast,2}_0$$ as complexes and $H^p(\mathscr{F}^\ast_2)\simeq H^p(\mathscr{F}^\ast_1)\oplus  H^p( E^{\ast,2}_0)$. We are left to compute the cohomology group of the complex $E^{\ast,2}_0$. In this instance, $ E^{1,2}_0=0$ and we are computing the cohomology of the sequence $0\to  E^{2,2}_0\to 0$.  Equivalently, we are left to identify the set $E^{2,2}_0$. 

Any $P(\mu)\in E^{2,2}_0$ is of the form $$P(\mu)=\sum_{i<j} \dfrac{dx_{ij}}{x_{ij}}\wedge\dfrac{dy_{ij}}{y_{ij}} \alpha_{ij}$$ for $\alpha_{ij}\in\Omega^{0}(Z_i\cap Z_j)$. In other words, each $\alpha_{ij}$ is a function defined on a set of discrete points. 

Then the set $E^{2,2}_0$ can be identified with $$\bigoplus_{i<j} \left\{\alpha_{ij}\in\Omega^{p-2}(Z_i\cap Z_j)\right\}.$$ By the computation which can be found in example 2.11 of \cite{Lanius}, these sets are invariant under change of $Z_i$ and $Z_j$ defining functions. Thus $$H^p(E^{\ast,2}_0)\simeq \bigoplus_{i<j} H^{p-2}(Z_i\cap Z_j)$$ and $$H^p(\mathscr{F}^\ast_2)\simeq H^p(S)\oplus\bigoplus_i H^{p-1}(Z_i)\oplus\bigoplus_{i<j} H^{p-2}(Z_i\cap Z_j).$$

\vspace{2ex}

In our final step, we will show that the cohomology of $E^{\ast,k}_0$ vanishes for $k\geq 3$, and thus $H^p(\mathscr{F}_2^\ast)$ actually computes ${}^bH^p(S)$.    

\vspace{1ex}

\noindent {\bf Computing $H^p(\mathscr{F}^\ast_k)$ for $k\geq 3$.}

Fix $k\geq 3$ and consider the short exact sequence $$0\to \mathscr{F}^\ast_{k-1}\to \mathscr{F}^\ast_{k}\to E^{\ast,k}_0\to 0.$$

Given any $p\in S$ such that at least the $k$ curves in the set $I=\left\{Z_1,\dots,Z_k\right\}$ intersect at $p$, there exists a star chart $(U_{p,I},x,y)$ centered at $p$ such that $$Z_1=\left\{x=0\right\},Z_2=\left\{y=0\right\},\dots,Z_k=\left\{A_kx+B_ky=0\right\}.$$  

Let $\chi_{p,I}$ be a smooth bump function  supported in $U_{p,I}$ and with constant value one in a neighborhood of $p$. 

 Any element $\mu \in \mathscr{F}_k^1$ can be expressed as  $$\mu=\sum_{\begin{array}{c} p, I\\ \scalebox{.9}{|I|=k}\end{array}}\chi_{p,I}\cdot\prod_{i=3}^k\dfrac{1}{(A_ix+B_iy)}\left(\dfrac{dx}{x}-\dfrac{dy}{y}\right)\wedge\alpha_{p,I}+\beta $$ for $\alpha_{p,I}\in\Omega^{0}(\left\{p\right\})$ and $\beta\in\mathscr{F}_{k-1}^1$. Any element $\mu \in \mathscr{F}_k^2$ can be expressed as  $$\mu=\sum_{\begin{array}{c} p, I\\ \scalebox{.9}{|I|=k}\end{array}}\chi_{p,I}\cdot\prod_{i=3}^k\dfrac{1}{(A_ix+B_iy)}\dfrac{dx\wedge dy}{xy}\wedge\alpha_{p,I}+\beta $$ for $\alpha_{p,I}\in\Omega^{0}(\left\{p\right\})$ and $\beta\in\mathscr{F}_{k-1}^2$. Note that each $\alpha_{p,I}$ is simply a real number. 

Given $\mu\in \mathscr{F}_k^\ell$, we write $\mathfrak{R}(\mu)=\beta$ and $\mathfrak{S}(\mu)=\mu-\mathfrak{R}(\mu)$ for `regular' and `singular' parts. One can check that $\mathfrak{R}(d\mu)=d(\mathfrak{R}(\mu))$ and $\mathfrak{S}(d\mu)=d(\mathfrak{S}(\mu))$. Thus we have a  splitting $$\mathscr{F}^\ast_k=\mathscr{F}^\ast_{k-1}\oplus E^{\ast,k}_0$$ as complexes and $H^p(\mathscr{F}^\ast_k)\simeq H^p(\mathscr{F}^\ast_{k-1})\oplus  H^p(E^{\ast,k}_0)$. We are left to compute the cohomology group of the  complex $E^{\ast,k}_0$. 

We will first compute the cohomology at $E^{1,k}_0$. 
Any $P(\mu)\in E^{1,k}_0$ is of the form $$P(\mu)=\sum_{\begin{array}{c} p, I\\ \scalebox{.9}{|I|=k}\end{array}}\chi_{p,I}\cdot\prod_{i=3}^k\dfrac{1}{(A_ix+B_iy)}\left(\dfrac{dx}{x}-\dfrac{dy}{y}\right)\wedge\alpha_{p,I}$$ for $\alpha_{p,I}\in\Omega^{0}(\left\{p\right\})$. Then, as an element of $E^{\ast,k}_0$, \begin{equation}\label{eq:2} dP(\mu)= \sum_{\begin{array}{c} p, I\\ \scalebox{.9}{|I|=k}\end{array}}\chi_{p,I}\cdot\prod_{i=3}^k\dfrac{1}{(A_ix+B_iy)}\dfrac{dx\wedge dy}{xy}\wedge(k-2)\alpha_{p,I}.\end{equation}

Thus $\ker (d:E^{1,k}_0\to E^{2,k}_0)=\left\{\alpha_{p,I}=0\right\}$ and $H^1(E^{\ast,k}_0)=0$ for all $k\geq 3$. \vspace{1ex}

Next, we compute the cohomology at $E^{2,k}_0$. Any element  $P(\mu)\in E^{2,k}_0$ is of the form $$P(\mu)=\sum_{\begin{array}{c} p, I\\ \scalebox{.9}{|I|=k}\end{array}}\chi_{p,I}\cdot\prod_{i=3}^k\dfrac{1}{(A_ix+B_iy)}\dfrac{dx\wedge dy}{xy}\wedge\alpha_{p,I}$$ for $\alpha_{p,I}\in\Omega^{0}(\left\{p\right\})$. We have $dP(\mu)=0$ so $\ker (d:E^{2,k}_0\to 0)=E^{2,k}_0$. By equation (\ref{eq:2}), $\text{im}(d:E^{1,k}_0\to E^{2,k}_0)=E^{2,k}_0$. As desired, we have shown that $H^2(E^{\ast,k}_0)=0$ for $k\geq 3$ and have completed the proof of the theorem.  \end{proof}

\subsection{\bf A log Darboux theorem} 

Darboux's theorem gives us a local description of Log Poisson structures which enables us to identify the rigged algebroid and compute Poisson cohomology. We establish this type of normal form theorem using a Moser-type argument in a neighborhood of $p\in D$. Let $(S,D)$ be a surface $S$ with star divisor $D$ and assume $p\in D$ has degree $k$. In a star chart centered at $p$, consider the star log symplectic form

\begin{equation}\label{eq:normal} \omega_0=  \left(\dfrac{1}{xy}\prod_{i=3}^k\dfrac{1}{(A_ix+B_iy)}+P\right) dx\wedge dy \end{equation} where $$P=\sum_{i}\dfrac{\lambda_i}{\text{\scalebox{.9}{$x(A_ix+B_iy)$}}}+\dfrac{\delta_i}{\text{\scalebox{.9}{$y(A_ix+B_iy)$}}}+\sum_{i<j}\dfrac{\nu_{ij}}{\text{\scalebox{.9}{$(A_ix+B_iy)(A_jx+B_jy)$}}}$$ for real numbers $\lambda_i, \delta_i, \nu_{ij}$.  

\begin{proposition}\label{prop01} Let $\omega$ be a star log symplectic form on $(S,D)$. Given a degree $k$ point $p\in D$, there exists a neighborhood $U$ of $p$ such that on $U$ there is a b-symplectomorphism pulling $\omega$ back to  (\ref{eq:normal}). 
\end{proposition} 

\begin{proof} 
 Let $(S,D)$ be a surface $S$ with star divisor $D$ and assume $p\in D$ has degree $k$. Let $\omega$ be a log-symplectic structure on $S$. The proof of Theorem \ref{bstarcohom} tells us that at any intersection point $p\in D$, there exists a star chart $U$ such that $\omega$ is expressible as $\omega_0+\nu$ where $\nu\in{}^b\Omega^2(U)$ is an exact log 2-form and $\omega_0$ is the log 2-form given by equation (\ref{eq:normal}). Further,  $$\omega|_p= \dfrac{1}{xy}\prod_{i=3}^k\dfrac{1}{(A_ix+B_iy)}dx\wedge dy$$ and  $\nu|_p=0$. We have that $$\omega-\omega_0=\nu=d\gamma$$ for some $\gamma\in{}^b\Omega^1(U)$. Note that $\gamma$ is not a primitive for any terms of the form  $$\dfrac{1}{xy}\prod_{i=3}^k\dfrac{c}{(A_ix+B_iy)}dx\wedge dy$$ for $c\in\mathbb{R}$. Consequently, when considered as a b-form, $\gamma$ vanishes at the point $p$ because it is not `maximally singular'. 
 
Now we will proceed by the standard relative Moser argument (See \cite{Cannas} Sec. 7.3 for the smooth setting, and \cite{Scott} Thm 6.4 for the $b$-symplectic version). Choose real numbers $\lambda_i,\delta_i,\nu_{ij}$ so that $\omega$ and $\omega_0$ are cohomologous in the Lie algebroid cohomology of the $b$-tangent bundle. Further, if $c>0$, choose $h=1$. If $c<0$, choose $h=-1$. This second condition means that $\omega$ and $\omega_0$ induce the same orientation of the $b$-tangent bundle. It is necessary to ensure that the convex combination $\omega_t=(1-t)\omega_0+t\omega$ is non-degenerate for all time $t$. Then $\dfrac{d\omega_t}{d t}=\omega-\omega_0=d\gamma.$  Because $\gamma$ is a log one form, the vector field defined by $i_{v_t}\omega_t=-\gamma$ is a log vector field and its flow fixes the divisor $D\cap U$. Thus we can integrate $v_t$ to an isotopy that fixes $D\cap U$. This isotopy is the desired $\log$-symplectomorphism. \end{proof} 

%The details proving this proposition can be found in \cite{LaniusPartitionable} Prop. 2.2 where we establish a version for log-symplectic structures over a normal crossing divisor. The argument is also quite similar to the standard relative Moser argument which can be found in \cite{Cannas} Sec. 7.3 for the smooth setting or \cite{Scott} Thm. 6.4 for the $b$-symplectic version.   

\subsection{\bf Global Moser for star log symplectic manifolds}

\begin{proposition}\label{globalmoser} Let $S$ be a compact surface with a star divisor $D$ and let $\omega_0$ and $\omega_1$ be two star log symplectic forms on $(S,D)$. Suppose that there is a family of star log symplectic forms $\omega_t$ from $\omega_0$ to $\omega_1$ defined for $0\leq t\leq 1$. If the $b$-cohomology class $[\omega_t]$ is independent of $t$, then there exists a family of diffeomorphisms $$\gamma_t:S\to S\text{ for }0\leq t\leq 1$$ such that $$ \gamma_t|_D\text{ is the identity map on }D\text{ and }\gamma_t^*\omega_t=\omega_0.$$ \end{proposition}

Because the proof of Proposition \ref{globalmoser} is very similar to the proof of Proposition \ref{prop01} and follows closely the proofs of Theorem 38 in \cite{Guillemin01} or Theorem 6.5 in \cite{Scott}, we will omit the details. This log version of the global Moser theorem completes the proof of Theorem \ref{thmclass} and gives us a classification of star log symplectic surfaces by $b$-de Rham cohomology classes.

\section{\bf Poisson Cohomology}\label{section4}  
 To compute the Poisson cohomology of star $\log$ symplectic surfaces, we construct a Lie algebroid that is isomorphic to the Poisson Lie algebroid of $\pi$.    

\begin{definition} Let $(S,D,\omega)$ be a star $\log$ symplectic surface of the $b$-tangent bundle $({}^bTS,\rho:{}^bTS\to TS)$ over $(S,D)$. The \emph{rigged Lie algebroid of $\omega$} is the vector bundle whose space of sections is $$\left\{u\in\mathcal{C}^\infty(S;{}^bTS)|i_u\omega\in\rho^*(T^*S)\right\},$$ i.e. the $b$-vector fields that when contracted into $\omega$ ``smoothen" it into a smooth one form. Alternatively, the sections of the dual rigged bundle  $\Gamma(\mathcal{R}^*)$ are an extension to $S$ of the image $\omega^\flat(\Gamma(TS))$ away from $D$.\end{definition} 

As discussed in \cite{LaniusPartitionable}, the rigged Lie algebroid is isomorphic to the Poisson Lie algebroid $T^*S$ with anchor map $\pi^\sharp=(\omega^\flat)^{-1}$. However, completing the computation of cohomology using the Lie algebroid cohomology of the rigged Lie algebroid is much more tractable because it is a complex of forms rather than multi-vectorfields.  

\subsection{\bf Identifying the rigged Lie algebroid} 

Next, we show that the rigged Lie algebroid of a star log Poisson surface is in fact a familiar Lie algebroid, called the zero tangent bundle. 

\begin{lemma} The Poisson cohomology of a star log Poisson surface $(S,D,\pi)$ is isomorphic to the de Rham cohomology ${}^0H^\ast(S)$ of the zero tangent bundle of $(S,D)$. \end{lemma} 

\begin{proof} The proof that the rigged Lie algebroid computes Poisson cohomology can be found in Section 5 of \cite{Lanius}. Thus it suffices to show that the rigged Lie algebroid is isomorphic to the zero-tangent bundle. 

Recall that in a neighborhood of an intersection point $p$ of degree $k\geq 3$, $\omega$ can be expressed as $$\omega_0=  \left(\dfrac{1}{xy}\prod_{i=3}^k\dfrac{1}{(A_ix+B_iy)}+P\right) dx\wedge dy$$ where $$P=\sum_{i}\dfrac{\lambda_i}{\text{\scalebox{.9}{$x(A_ix+B_iy)$}}}+\dfrac{\delta_i}{\text{\scalebox{.9}{$y(A_ix+B_iy)$}}}+\sum_{i<j}\dfrac{\nu_{ij}}{\text{\scalebox{.9}{$(A_ix+B_iy)(A_jx+B_jy)$}}}$$ for real numbers $\lambda_i, \delta_i, \nu_{ij}$. Then the space $$\left\{u\in\mathcal{C}^\infty(S;{}^bTS)|i_u\omega\in\rho^*(T^*S)\right\}$$ is locally generated by $$xy\prod_{i=3}^k(A_ix+B_iy)\dfrac{\partial}{\partial x},\quad xy\prod_{i=3}^k(A_ix+B_iy)\dfrac{\partial}{\partial y}.$$ 

These are local generators of the zero tangent bundle, that is the vector bundle whose space of sections are $$\left\{u\in\mathcal{C}^\infty(S;TS)|u|_Z=0 \text{ for all }Z\in D\right\}.$$ This vector bundle, first introduced by Rafe Mazzeo and Richard Melrose in the context of manifolds with boundary \cite{Mazzeo, MazzeoMelrose}, is a Lie algebroid. The anchor map is evaluation in the tangent bundle and the Lie bracket is induced by the standard Lie bracket on $TS$. \end{proof}

\subsection{\bf $0$-de Rham cohomology} 

Given a surface $S$ with a collection $D$ of transverse curves $\left\{Z_1,\dots,Z_k\right\}$, the de Rham cohomology of the associated $0$-tangent bundle is  \begin{equation}\label{eq:eq10}\text{\scalebox{.8}{$\displaystyle{H^p(S)\oplus \text{\scalebox{.8}{$\bigoplus_{i}$}}H^{p-1}(Z_i)\oplus\text{\scalebox{.8}{$\bigoplus_{i<j}$}}\left(H^{p-2}\left(Z_i\cap Z_j)\oplus H^{p-2}(Z_i\cap Z_j;|N^*Z_i|^{-1}\otimes|N^*Z_j|^{-1}\right)\right)}$ }}\end{equation} which we originally computed in \cite{LaniusPartitionable}. Unlike in the case of the $b$-tangent bundle, this decomposition does not hold in the more general case when $D$ is any star divisor. In other words, the $0$-de Rham cohomology perceives higher order intersection of curves in $D$ while $b$-de Rham does not.   

\begin{theorem} \label{0starcohom}
Let $(S,D)$ be a surface $S$ with ordered star divisor $$D=\left\{Z_1,\dots,Z_k\right\}.$$ Given a star atlas with respect to this ordering, the Lie algebroid cohomology of the 0-tangent bundle over $(S,D)$ is isomorphic to the expression (\ref{eq:eq10}) in degree 0 and 1. In degree 2, the cohomology is a direct sum of the following vector spaces: 

\begin{itemize}
    \item A single copy of $H^2(S)$.  \vspace{1ex}
    
    \item Each hypersurfaces $Z_i$ contributes $H^1(Z_i)$. \vspace{1ex}
    
    \item Each pair wise intersection of two hypersurfaces $Z_i,Z_j$ with $i<j$ contributes \vspace{1ex}
    
     \centerline{\scalebox{.9}{$\displaystyle{H^{0}\left(Z_i\cap Z_j)\oplus H^{0}(Z_i\cap Z_j;|N^*Z_i|^{-1}\otimes|N^*Z_j|^{-1}\right).}$}}\vspace{1ex} 
    
    \item Each intersection of three hypersurfaces $Z_i,Z_j,Z_k$ with $i<j<k$ contributes \vspace{1ex}
    
    \centerline{\scalebox{.875}{$\displaystyle{H^{0}(Z_i\cap Z_j\cap Z_k;|N^*Z_i|^{-1}\otimes|N^*Z_j|^{-1}\otimes|N^*Z_k|^{-1})\oplus}$}} \vspace{1ex}
    
\centerline{\scalebox{.875}{$\displaystyle{H^{0}(Z_i\cap Z_j\cap Z_k;|N^*Z_j|^{-1}\otimes|N^*Z_k|^{-1})\oplus H^{0}(Z_i\cap Z_j\cap Z_k;|N^*Z_i|^{-1}\otimes|N^*Z_k|^{-1})}.$}} \vspace{1ex}

\item Each intersection of four or more hypersurfaces $Z_{i_1},Z_{i_2},Z_{i_3},\dots,Z_{i_\ell}$ with $i_1<i_2<i_3<\dots <i_\ell$ contributes \vspace{1ex}

\centerline{\scalebox{.9}{$\displaystyle{H^{0}(\text{\scalebox{.7}{$\bigcap_i$}} Z_i;\text{\scalebox{.7}{$\bigotimes_i$}}|N^*Z_i|^{-1})\oplus H^{0}(\text{\scalebox{.7}{$\bigcap_i$}} Z_i;\text{\scalebox{.7}{$\bigotimes_{i\neq i_1}$}}|N^*Z_i|^{-1})}\oplus$}}\vspace{1ex}

\centerline{\scalebox{.9}{$\displaystyle{  H^{0}(\text{\scalebox{.7}{$\bigcap_i$}} Z_i;\text{\scalebox{.7}{$\bigotimes_{i\neq i_2}$}}|N^*Z_i|^{-1})\oplus H^{0}(\text{\scalebox{.7}{$\bigcap_i$}} Z_i;\text{\scalebox{.7}{$\bigotimes_{i\neq i_1,i_2}$}}|N^*Z_i|^{-1})}$.}}

\end{itemize}

\vspace{1ex}

\end{theorem}

\begin{Example}  
Consider the star log symplectic structure on the sphere introduced in Example \ref{motivatingexample}. The theorem tells us that in fixed coordinates, we can identify $$H_\pi^2(\mathbb{S}^2)\simeq \mathbb{R}^{22}.$$

\end{Example}

Further, in the following example we demonstrate how many of the classes in $H_\pi^2(S)$ do not arise as bi-vectors on the $b$-tangent bundle. 

\begin{Example}\label{NonbClass} Consider $\mathbb{R}^2$ with standard coordinates equipped with Poisson bi-vector $$\pi=(x+y)(x-y)x\dfrac{\partial}{\partial x}\wedge\dfrac{\partial}{\partial y}.$$ The degeneracy locus of $\pi$ is depicted in Figure \ref{fig1}. As described in Theorem \ref{PC}, the second Poisson cohomology $H^2_\pi(\mathbb{R}^2)$ contains three copies of $$H^0(\left\{x=0\right\}\cap \left\{x-y=0\right\}\cap\left\{x+y=0\right\}).$$ One of these vector spaces is generated by the bi-vector $$\nu=x\dfrac{\partial}{\partial x}\wedge\dfrac{\partial}{\partial y}.$$ Since $\nu$ does not vanish at the lines $x+y=0$ or $x-y=0$, it does not correspond to an element of the $b$-de Rham sub-complex. Further, its corresponding element in the rigged algebroid $\mathcal{R}$ is not cohomologous to any $b$-de Rham form. Thus not all non-trivial deformations of $\pi$ near $\pi$ can be done through a path of $\log$-symplectic structures.  \end{Example}

\begin{proof}

As in the proof of Theorem \ref{bstarcohom}, we will proceed using a spectral sequence. 

\noindent {\bf Constructing a filtration ${}^0\mathcal{F}$.} 

Let $(S,D)$ be a surface $S$ with star divisor $D$.  Given a subset $I\subseteq D$, let ${}^0\Omega^\ast(S,I)$ denote the complex of 0-de Rham forms of the 0-tangent bundle over $(S,I)$. In this notation, we want to compute the cohomology of the complex ${}^0\Omega^\ast(S,D)$. We filter the differential complex $({}^0\Omega^\ast(S),d)$ using a filtration denoted  $${}^0\mathcal{F}=\left\{\Omega^\ast(S)={}^0\mathscr{F}^\ast_0\subseteq {}^0\mathscr{F}^\ast_1\subseteq {}^0\mathscr{F}^\ast_2\subseteq\dots\subseteq {}^0\mathscr{F}^\ast_{|D|-1}\subseteq {}^0\mathscr{F}^\ast_{|D|}={}^0\Omega^\ast(S,D)\right\}$$ where $${}^0\mathscr{F}^*_k:={}^{v.s.}\text{Span}\left\{\bigcup_{\text{\scalebox{.8}{$I\subseteq D,|I|=k$}}}\hspace{-1ex}{}^0\Omega^\ast(S,I)\right\}.$$Notice that ${}^0\mathscr{F}^\ast_k\subseteq {}^0\mathscr{F}^\ast_{k+1}$ for all $k$. Each set ${}^0\mathscr{F}^p_k$ is a group under addition and has the structure of a module with scalars from the ring $${}^0\mathscr{F}^{0}_k=\mathcal{C}^\infty(S).$$ 

Note that ${}^0\Omega^\ast(S,I)\subseteq {}^0\Omega^\ast(S,D)$ for all $I\subseteq D$. Consequently, the modules ${}^0\mathscr{F}_k^\ast\subseteq {}^b\Omega^\ast(S,D)$ for all $k$. It is easy to check that ${}^0\mathscr{F}_k^\ast$ is closed under the differential of ${}^0\Omega^\ast(S,D)$. Thus we can define a differential ${}^0\mathscr{F}_k^p\xrightarrow{~d~}{}^0\mathscr{F}_k^{p+1}$ inherited from the differential on ${}^0\Omega^\ast(S,D)$. 

For each $k$, we have a chain complex: $$0\to {}^0\mathscr{F}^0_k\xrightarrow{~d~}{}^0\mathscr{F}^1_k\xrightarrow{~d~}{}^0\mathscr{F}^2_k\to 0.$$

\noindent {\bf The associated graded complex of ${}^0\mathcal{F}$.} 

Next, we will consider the associated graded complex of our filtered complex $({}^0\Omega^\ast(S),d,{}^0\mathcal{F})$. We denote the associated graded module $({}^0\mathscr{F}^p_{k-1})/({}^0\mathscr{F}^p_k)$ by $E^{p,k}_0$. 

The inclusions ${}^0\mathscr{F}^\ast_{k-1}\to {}^0\mathscr{F}^\ast_{k}$ of cochain complexes from filtration ${}^0\mathcal{F}$ fit into a short exact sequence of complexes \begin{center} $0\to {}^0\mathscr{F}^\ast_{k-1}\to {}^0\mathscr{F}^\ast_{k}\to E^{\ast,k}_0\to 0$. \end{center} The differential $^{E}d$ on $E^{p,k}_0$ is induced by $d$ on ${}^0\mathscr{F}^p_{k}$: if $P$ is the projection ${}^0\mathscr{F}^p_k\to{}^0\mathscr{F}^p_k/{}^0\mathscr{F}^p_{k-1},$ then  $^{E}d(\eta)=P(d(\theta))$ where  $\theta\in{}^0\mathscr{F}^p_k$ is any form such that $P(\theta)=\eta$. Hence $({}^Ed)^2=0$ and $(E^{\ast,k}_0,{}^Ed)$ is a complex. \vspace{2ex}

The associated graded complex of ${}^0\Omega^\ast(S),d)$ is $\bigoplus_{k=1}^m E^{\ast,k}_0$. The following computation will show that the spectral sequence collapses at the first page and $$\bigoplus_{k=1}^m H^d(E^{\ast,k}_0)\simeq {}^0H^d(S).$$ \vspace{2ex}

\noindent {\bf Computing $H^p({}^0\mathscr{F}^\ast_1)$.}

Let $k=1$ and consider the short exact sequence $$0\to {}^0\mathscr{F}^\ast_{0}\to {}^0\mathscr{F}^\ast_{1}\to E^{\ast,1}_0\to 0$$ defined above. Note that $\mathscr{F}^\ast_0=\Omega^\ast(S)$. 

Let $D=\left\{Z_1,Z_2,\dots,Z_\ell\right\}$ be our star divisor and consider defining functions $x_1,x_2,\dots,x_\ell$ such that $Z_i=\left\{x_i=0\right\}$. For each $i$, let $\chi_{i}:S\to\mathbb{R}$ be a smooth bump function  supported in a neighborhood of $Z_i$ and with constant value one near $Z_i$. An element $\mu \in \mathscr{F}_1^1$ can be expressed as $$\mu=\sum_{i=1}^\ell \chi_i\cdot\left(\dfrac{dx_i}{x_i}\alpha_i+\dfrac{\beta_i}{x_i}\right)+\theta$$ where $\alpha_i\in\mathcal{C}^\infty(Z_i)$, $\beta_i\in\Omega^1(Z_i)$, and $\theta\in\Omega^1(S)$.

Note that $$d\mu=-\sum_{i=1}^\ell \chi_i\left(\dfrac{dx_i}{x_i}\wedge d\alpha_i-\dfrac{dx_i\wedge\beta_i}{x_i^2}\right)+d\chi_i\wedge\left(\dfrac{dx_i}{x_i}\alpha_i+\dfrac{\beta_i}{x_i}\right) +\dfrac{d\beta_i}{x_i}+d\theta.$$ 

An element $\mu \in \mathscr{F}_1^2$ can be expressed as $$\mu=\sum_{i=1}^\ell \chi_i\cdot \dfrac{dx_i}{x_i^2}\wedge(\alpha_{i}+\beta_{i}x_i)+\theta$$ where $\alpha_i,\beta_i\in\Omega^1(Z_i)$ and $\theta\in\Omega^2(S)$. Note that $d\mu=0$. 

Given $\mu\in \mathscr{F}_1^p$, we write $\mathfrak{R}(\mu)=\theta$ and $\mathfrak{S}(\mu)=\mu-\mathfrak{R}(\mu)$ for `regular' and `singular' parts. It is easy to see that $\mathfrak{R}(d\mu)=d(\mathfrak{R}(\mu))$ and $\mathfrak{S}(d\mu)=d(\mathfrak{S}(\mu))$. Thus we have a  splitting $${}^0\mathscr{F}^\ast_1=\Omega^*(S)\oplus E^{1,*}_0$$ as complexes and $H^p({}^0\mathscr{F}^\ast_1)\simeq H^p(S)\oplus  H^p(E^{1,*}_0)$. We are left to compute the cohomology group of the quotient complex $E^{1,*}_0$. 

Any $P(\mu)\in E_0^{\ast,1}$ is of the form $$P(\mu)=\sum_{i=1}^\ell \dfrac{dx_i}{x_i}\alpha_i+\dfrac{\beta_i}{x_i}$$ for $\alpha_i\in\mathcal{C}^\infty(Z_i)$, $\beta_i\in\Omega^1(Z_i)$. Then $$dP(\mu)=-\sum_{i=1}^\ell \dfrac{dx_i}{x_i}\wedge d\alpha_i-\dfrac{dx_i\wedge\beta_i}{x_i^2}+\dfrac{d\beta_i}{x_i}.$$ 
By inspecting the expression at each $Z_i$, this gives us kernel relations $d\alpha_i=0$, $\beta_i=0$. Given these choices of $Z_i$ defining functions $x_1,\dots,x_\ell$, we identify the $\ker (d: E_0^{p,1} \to E_0^{p+1,1})$ as $$\bigoplus_{i=1}^\ell \left\{\alpha_i\in\mathcal{C}^\infty(Z_i):d\alpha_i=0\right\}.$$  Thus $$H^1(E_0^{\ast,1})\simeq \bigoplus_{i} H^0(Z_i).$$

For $H^2(E_0^{\ast,1})$, note that given $\mu\in{}^0\mathscr{F}^2_1$ there exists $\tilde{\mu}\in{}^0\mathscr{F}^1_1$ such that $$P(\mu-d\tilde{\mu})=\sum_i\dfrac{dx_i}{x_i}\wedge(\beta_i-d\alpha_i)$$ for $\beta_i\in\Omega^1(Z_i)$, $\alpha_i\in\Omega^0(Z_i)$, and that elements of this form are not in the image of $d$. Thus we can identify $$H^2(E_0^{\ast,1})\simeq \bigoplus_i H^1(Z_i).$$ 

By a computation, which can be found in example 2.11 of \cite{Lanius}, these sets are invariant under change of $Z_i$ defining function. Thus $$H^p(E_0^{\ast,1})\simeq \bigoplus_i H^{p-1}(Z_i)\text{ and }H^p({}^0\mathscr{F}^\ast_1)\simeq H^p(S)\oplus\bigoplus_i H^{p-1}(Z_i).$$

\vspace{2ex}

\noindent {\bf Computing $H^p({}^0\mathscr{F}^\ast_2)$.} 

Let $k=2$ and consider the short exact sequence $$0\to {}^0\mathscr{F}^\ast_{1}\to {}^0\mathscr{F}^\ast_{2}\to E_0^{\ast,2}\to 0.$$

Let $D=\left\{Z_1,Z_2,\dots,Z_\ell\right\}$ be our star divisor. For each pair of curves $Z_i, Z_j$ with $i<j$, there is a pair of tubular neighborhoods $\tau_i=Z_i\times(-\varepsilon, \varepsilon)_{x_{ij}}$ of $Z_i$ and $\tau_j=Z_j\times(-\varepsilon, \varepsilon)_{y_{ij}}$ of $Z_j$ such that $$\left[\dfrac{\partial}{\partial x_{ij}},\dfrac{\partial}{\partial y_{ij}}\right]=0.$$ For existence of such neighborhoods see, for instance, section 5 of \cite{Albin}.  

Let $\chi_{ij}$ be a smooth bump function supported near $Z_i\cup Z_j$ and with constant value one near $Z_i\cup Z_j$. An element $\mu\in {}^0\mathscr{F}^1_2$ can be expressed as $$\mu = \sum_{i<j}\chi_{ij}\cdot\left(\dfrac{dx_{ij}}{x_{ij}y_{ij}}\alpha_{ij}+\dfrac{dy_{ij}}{x_{ij}y_{ij}}\beta_{ij}\right) +\theta $$ for $\alpha_{ij},\beta_{ij}\in\mathcal{C}^\infty(Z_i\cap Z_j)$, and $\theta\in{}^0\mathscr{F}^1_1$. Note that $\alpha$ and $\beta$ are just constants.  Thus $$d\mu = \sum_{i<j}\dfrac{dx_{ij}\wedge dy_{ij}}{x_{ij}y^2_{ij}}\alpha_{ij}-\dfrac{dx_{ij}\wedge dy_{ij}}{x^2_{ij}y_{ij}}\beta_{ij} +d\theta.$$

An element $\mu\in {}^0\mathscr{F}^2_2$ can be expressed as $$\mu = \sum_{i<j}\dfrac{dx_{ij}\wedge dy_{ij}}{x^2_{ij}y^2_{ij}}(\alpha_{ij}+\beta_{ij}x_{ij}+\gamma_{ij}y_{ij}+\delta_{ij}x_{ij}y_{ij}) +\theta $$ for $\alpha_{ij},\beta_{ij},\gamma_{ij},\delta_{ij}\in\mathcal{C}^\infty(Z_i\cap Z_j)$, and $\theta\in{}^0\mathscr{F}^2_1$. Note $d\mu=0$. 

Given $\mu\in \mathscr{F}_2^p$, we write $\mathfrak{R}(\mu)=\theta$ and $\mathfrak{S}(\mu)=\mu-\mathfrak{R}(\mu)$ for `regular' and `singular' parts. It is easy to see that $\mathfrak{R}(d\mu)=d(\mathfrak{R}(\mu))$ and $\mathfrak{S}(d\mu)=d(\mathfrak{S}(\mu))$. Thus we have a  splitting $${}^0\mathscr{F}^\ast_2={}^0\mathscr{F}^\ast_1\oplus E_0^{\ast,2}$$ as complexes and $H^p({}^0\mathscr{F}^\ast_2)\simeq H^p({}^0\mathscr{F}^\ast_1)\oplus  H^p(E_0^{\ast,2})$. We are left to compute the cohomology group of the quotient complex $E_0^{\ast,2}$. 

Any $P(\mu)\in E_0^{\ast,2}$ is of the form
$$P(\mu)=\sum_{i<j}\dfrac{dx_{ij}}{x_{ij}y_{ij}}\alpha_{ij}+\dfrac{dy_{ij}}{x_{ij}y_{ij}}\beta_{ij}$$ for $\alpha_{ij},\beta_{ij}\in\mathcal{C}^\infty(Z_i\cap Z_j)$. Then $$dP(\mu)=\sum_{i<j}\dfrac{dx_{ij}\wedge dy_{ij}}{x_{ij}y^2_{ij}}\alpha_{ij}-\dfrac{dx_{ij}\wedge dy_{ij}}{x^2_{ij}y_{ij}}\beta_{ij}.$$ Thus our kernel relations are $\alpha_{ij}=0$ and $\beta_{ij}=0$. 

For $H^2(E_0^{\ast,2})$, note that for $\mu\in {}^0\mathscr{F}^2_2$ there exists $\tilde{\mu}\in{}^0\mathscr{F}^1_2$ such that $$P(\mu-d\tilde{\mu})=\dfrac{dx_{ij}\wedge dy_{ij}}{x^2_{ij}y^2_{ij}}\wedge(\alpha_{ij} + \delta_{ij}x_{ij}y_{ij})$$ for $\alpha_{ij},\delta_{ij}\in\mathcal{C}^\infty(Z_i\cap Z_j)$. Further, there are no elements of this form in the image of $d$. Thus, given these choices $x_{ij},y_{ij}$ of defining functions, we can identify $$H^2(E_0^{\ast,2})\simeq \bigoplus_{i<j} H^0(Z_i\cap Z_j)\oplus H^0(Z_i\cap Z_j).$$ 

By a computation which can be found at the conclusion of the proof of Theorem 2.15 in \cite{Lanius}, this cohomology can be identified independently of defining functions as 
$$H^2(E_0^{\ast,2})\simeq \bigoplus_{i<j} H^0(Z_i\cap Z_j)\oplus H^0(Z_i\cap Z_j;|N^*Z_i|^{-1}\otimes |N^*Z_j|^{-1}).$$ Thus  $H^p({}^0\mathscr{F}^\ast_2)$ is $$ H^p(S)\oplus\bigoplus_i H^{p-1}(Z_i)\oplus  \bigoplus_{i<j} H^{p-2}(Z_i\cap Z_j)\oplus H^{p-2}(Z_i\cap Z_j;|N^*Z_i|^{-1}\otimes |N^*Z_j|^{-1}).$$

\vspace{1ex}

\noindent {\bf Computing $H^p({}^0\mathscr{F}^\ast_k)$ for $k\geq 3$.}

Fix $k\geq 3$ and consider the short exact sequence $$0\to {}^0\mathscr{F}^\ast_{k-1}\to {}^0\mathscr{F}^\ast_{k}\to E_0^{\ast,k}\to 0.$$

Given any $p\in S$ such that at least the $k$ curves in the set $I=\left\{Z_1,\dots,Z_k\right\}$ intersect at $p$, there exists a star chart with respect to $I$ centered at $p$, $(U_{p,I},x,y)$, such that $$Z_1=\left\{x=0\right\},Z_2=\left\{y=0\right\},\dots,Z_k=\left\{A_kx+B_ky=0\right\}.$$  

Let $\chi_{p,I}$ be a smooth bump function  supported in $U_{p,I}$ and with constant value one in a neighborhood of $p$. 

Let $$R:=\prod_{i=3}^k(A_ix+B_iy).$$ An element $\mu\in {}^0\mathscr{F}^1_k$ can be expressed as $$\mu = \sum_{\begin{array}{c} p, I\\ \scalebox{.8}{|I|=k}\end{array}}\chi_{p,I}\cdot\left(\dfrac{dx}{xyR}\alpha_{p,I}+\dfrac{dy}{xyR}\beta_{p,I}\right)+\theta $$ for $\alpha_{p,I},\beta_{p,I}\in\Omega^{0}(\left\{p\right\})$ (i.e. constants) and $\theta\in{}^0\mathscr{F}_{k-1}^{1}$. 
Note %\begin{equation}\label{eq:6} \begin{split} d\mu= \sum_{\begin{array}{c} p, I\\ \scalebox{.9}{|I|=k}\end{array}}\chi_{p,I}\cdot\bigg(&\dfrac{dx\wedge dy}{xy^2P}\alpha_{p,I}+\dfrac{dx\wedge dy}{xyP^2}(\partial_yP)\alpha_{p,I}\\ &-\dfrac{dx\wedge dy}{x^2yP}\beta_{p,I}-\dfrac{dx\wedge dy}{xyP^2}(\partial_xP)\beta_{p,I}\bigg)+d\theta \end{split}\end{equation}

\text{\scalebox{.8}{$ \displaystyle{d\mu= \sum_{\begin{array}{c} p, I \\ \scalebox{.9}{|I|=k}\end{array}}\chi_{p,I}\cdot\bigg(\dfrac{dx\wedge dy}{xy^2R}\alpha_{p,I}+\dfrac{dx\wedge dy}{xyR^2}(\partial_yR)\alpha_{p,I} -\dfrac{dx\wedge dy}{x^2yR}\beta_{p,I}-\dfrac{dx\wedge dy}{xyR^2}(\partial_xR)\beta_{p,I}\bigg)}$}} 

\hspace{5ex}\text{\scalebox{.8}{$\displaystyle{+\sum_{\begin{array}{c} p, I\\ \scalebox{.8}{|I|=k}\end{array}}d\chi_{p,I}\wedge\left(\dfrac{dx}{xyR}\alpha_{p,I}+\dfrac{dy}{xyR}\beta_{p,I}\right)+d\theta }$}}

Any element $\mu \in {}^0\mathscr{F}_k^2$ can be expressed as  $$\mu=\sum_{\begin{array}{c} p, I\\ \scalebox{.9}{|I|=k}\end{array}}\chi_{p,I}\cdot\dfrac{dx\wedge dy}{x^2y^2R^2}\wedge(\alpha_{p,I}+\beta_{p,I}x+\gamma_{p,I}y+\delta_{p,I}xy)+\theta$$ for $\alpha_{p,I},\beta_{p,I},\gamma_{p,I},\delta_{p,I}\in\Omega^{0}(\left\{p\right\})$ and $\theta\in{}^0\mathscr{F}_{k-1}^2$. Note $d\mu=0$. 

Given $\mu\in {}^0\mathscr{F}_k^\ell$, we write $\mathfrak{R}(\mu)=\theta$ and $\mathfrak{S}(\mu)=\mu-\mathfrak{R}(\mu)$ for `regular' and `singular' parts. One can check that $\mathfrak{R}(d\mu)=d(\mathfrak{R}(\mu))$ and $\mathfrak{S}(d\mu)=d(\mathfrak{S}(\mu))$. Thus we have a  splitting $${}^0\mathscr{F}^\ast_k={}^0\mathscr{F}^\ast_{k-1}\oplus E_0^{\ast,k}$$ as complexes and $H^p({}^0\mathscr{F}^\ast_k)\simeq H^p({}^0\mathscr{F}^\ast_{k-1})\oplus  H^p(E_0^{\ast,k})$. We are left to compute the cohomology group of the quotient complex $E_0^{\ast,k}$. 

We will first compute the cohomology at $E_0^{1,k}$. 
Any $P(\mu)\in E_0^{1,k}$ is of the form $$P(\mu)=\sum_{\begin{array}{c} p, I\\ \scalebox{.9}{|I|=k}\end{array}}\chi_{p,I}\cdot\left(\dfrac{dx}{xyR}\alpha_{p,I}+\dfrac{dy}{xyR}\beta_{p,I}\right)$$ for $\alpha_{p,I},\beta_{p,I}\in\Omega^{0}(\left\{p\right\})$. Then $dP(\mu)$ follows from our expression of $d\mu$ above. Thus we can identify the set $\ker (d:E_0^{1,k}\to E_0^{2,k})=\left\{\alpha_{p,I}=0, \beta_{p,I}=0\right\}$ and $H^1(E_0^{\ast,k})=0$ for all $k\geq 3$. \vspace{1ex}

Next, we compute the cohomology at $E_0^{2,k}$. Any element  $P(\mu)\in E_0^{2,k}$ is of the form $$P(\mu)=\sum_{\begin{array}{c} p, I\\ \scalebox{.9}{|I|=k}\end{array}}\chi_{p,I}\cdot\dfrac{dx\wedge dy}{x^2y^2R^2}\wedge(\alpha_{p,I}+\beta_{p,I}x+\gamma_{p,I}y+\delta_{p,I}xy)$$ for $\alpha_{p,I},\beta_{p,I},\gamma_{p,I},\delta_{p,I}\in\Omega^{0}(\left\{p\right\})$. We have $dP(\mu)=0$. Note that $d(E_0^{1,k})$ has empty intersection with elements of the form $$\dfrac{dx\wedge dy}{x^2y^2R^2}\wedge(\alpha_{p,I}+\beta_{p,I}x+\gamma_{p,I}y).$$

Using the expression of $d\mu$ for $\mu\in E_0^{1,k}$ given above, we are left to consider when does the following equality hold at the point $p$:  $$\delta_{p,I}=(\partial_yP)\alpha_{p,I}-(\partial_xP)\beta_{p,I}$$ for $\alpha_{p,I},\beta_{p,I},\delta_{p,I}\in\Omega^0(\left\{p\right\})$. For this equation to hold true, $P$ would need to be a single line $P=Ax+By$. Thus when $k\geq 4$, the term $$\sum_{\begin{array}{c} p, I\\ \scalebox{.9}{|I|=k}\end{array}}\chi_{p,I}\cdot\dfrac{dx\wedge dy}{x^2y^2R^2}\wedge\delta_{p,I}xy$$ is not in the image of $d(E_0^{1,k})$. 

Let us consider the case where $k=3$. In this instance, the numbers $\beta_{p,I}=0$ and $\alpha_{p,I}=\dfrac{\delta_{p,I}}{B}$ provide a solution to the equality and $$\sum_{\begin{array}{c} p, I\\ \scalebox{.9}{|I|=k}\end{array}}\chi_{p,I}\cdot\dfrac{dx\wedge dy}{x^2y^2R^2}\wedge\delta_{p,I}xy$$ is in the image of $d(E_0^{1,3})$. 

Thus, for these fixed $Z_i$ defining functions, we have identified
$$H^n(E_0^{\ast,3})\simeq \bigoplus_{p, |I|=3} \left(H^{n-2}(\left\{p\right\})\right)^3$$ 
 and, for $k\geq 4$, 
$$H^n(E_0^{\ast,k})\simeq \bigoplus_{p, |I|=k} \left(H^{n-2}(\left\{p\right\})\right)^4.$$ 

For completeness, we will consider what happens under change of defining functions. By a computation which can be found at the conclusion of the proof of theorem 2.15 in \cite{Lanius}, the cohomology $H^p(E_0^{\ast,k})$ can be identified independently of defining functions as \vspace{1ex}

\centerline{\scalebox{.9}{$\displaystyle{ \text{\scalebox{.9}{$\bigoplus_{\text{\scalebox{.8}{$\begin{array}{c} i\in J\\ |J|=3 \end{array}$}}}$}}\big(H^{p-2}(\text{\scalebox{.7}{$\bigcap_i$}} Z_j;\text{\scalebox{.7}{$\bigotimes_i$}}|N^*Z_i|^{-1})\oplus H^{p-2}(\text{\scalebox{.7}{$\bigcap_i$}} Z_i;\text{\scalebox{.7}{$\bigotimes_{i\in J\setminus \left\{i_1\right\}}$}}|N^*Z_i|^{-1})\oplus H^{p-2}(\text{\scalebox{.7}{$\bigcap_i$}} Z_i;\text{\scalebox{.7}{$\bigotimes_{i\in J\setminus \left\{i_2\right\}}$}}|N^*Z_i|^{-1})\big)}$}} 

\noindent when $k=3$ and as \vspace{1ex}

\centerline{\scalebox{.9}{$\displaystyle{ \text{\scalebox{.9}{$\bigoplus_{\text{\scalebox{.8}{$\begin{array}{c} i\in J\\ |J|=k \end{array}$}}}$}}\big(H^{p-2}(\text{\scalebox{.7}{$\bigcap_i$}} Z_j;\text{\scalebox{.7}{$\bigotimes_i$}}|N^*Z_i|^{-1})\oplus H^{p-2}(\text{\scalebox{.7}{$\bigcap_i$}} Z_i;\text{\scalebox{.7}{$\bigotimes_{i\in J\setminus \left\{i_1\right\}}$}}|N^*Z_i|^{-1})\big)}$}}

\centerline{\scalebox{.9}{$\displaystyle{\oplus \text{\scalebox{.9}{$\bigoplus_{\text{\scalebox{.8}{$\begin{array}{c} i\in J\\ |J|=k\end{array}$}}}$}} \big(H^{p-2}(\text{\scalebox{.7}{$\bigcap_i$}} Z_i;\text{\scalebox{.7}{$\bigotimes_{i\in J\setminus \left\{i_1\right\}}$}}|N^*Z_i|^{-1})\oplus H^{p-2}(\text{\scalebox{.7}{$\bigcap_i$}} Z_i;\text{\scalebox{.7}{$\bigotimes_{i\in J\setminus \left\{i_1,i_2\right\}}$}}|N^*Z_i|^{-1})\big)}$}}\vspace{1ex}

\noindent when $k\geq 4$, where $J$ is a subset of integers $\left\{1,\dots,k\right\}$ and $i_1,i_2$ are the smallest elements in $J$.

\end{proof}

\end{document}